\documentclass[10pt]{amsart}
\usepackage{amsmath}
\usepackage{amsfonts}
\usepackage{amssymb}
\usepackage{graphicx}
\usepackage{amsthm,graphicx,color,yfonts}
\usepackage{pdfsync}
\usepackage{epstopdf}
\usepackage{epsfig}
\usepackage{marginnote} 
\usepackage{esint} 
\usepackage[colorlinks=true]{hyperref}
\hypersetup{linkcolor=red,citecolor=blue,filecolor=dullmagenta,urlcolor=darkblue} % coloured links

\numberwithin{equation}{section}

\theoremstyle{plain}
\newtheorem{theorem}{Theorem}[section]

\newtheorem{lemma}[theorem]{Lemma}
\newtheorem*{de-lemma}{Lemma}

\theoremstyle{remark}

\theoremstyle{definition}

\newcommand{\dd}{\mathrm{d}}
\newcommand{\R}{\mathbb{R}}

\newcommand{\ve}{\epsilon}

\begin{document}

\title{Gradient theory of domain walls in thin, nematic liquid crystals films}

\author{Marcel G. Clerc}
\address{Departamento de Física and Millennium Institute for Research in Optics, 
Facultad de Ciencias Físicas y Matemáticas, Universidad de Chile, Casilla, 487-3, Santiago, Chile.}
\email{marcel@dfi.uchile.cl}
\thanks{M. G. Clerc thanks for the financial support of FONDECYT project 1180903 and Millennium Institute for Research in Optics.}

\author{Micha{\l } Kowalczyk}
\address{Departamento de Ingenier\'{\i}a Matem\'atica and Centro
de Modelamiento Matem\'atico (UMI 2807 CNRS), Universidad de Chile, Casilla
170 Correo 3, Santiago, Chile.}
\email {kowalczy@dim.uchile.cl}
\thanks{M. Kowalczyk was partially supported by Chilean research grants Fondecyt 1130126 and 1170164, and Fondo Basal AFB170001 CMM-Chile.}

\author{Panayotis Smyrnelis}
\address{Centro
de Modelamiento Matem\'atico (UMI 2807 CNRS), Universidad de Chile, Casilla
170 Correo 3, Santiago, Chile.}
\email{psmyrnelis@dim.uchile.cl}
\thanks{P. Smyrnelis was partially supported by and Fondo Basal AFB170001 CMM-Chile and Fondecyt postdoctoral grant 3160055}

%\date{\today}

\begin{abstract}
In this paper we describe domain walls appearing in a thin, nematic  liquid crystal sample subject to an external field with intensity close to the  Fr\'eedericksz transition threshold. Using the gradient theory of the phase transition adopted to this situation, we show that depending on the parameters of the system, domain walls occur in the bistable region or at the border between the bistable and the monostable region. 
\end{abstract}

\maketitle

\section{Introduction}

The macroscopic systems maintained out of equilibrium by means of the injection and dissipation of energy 
are characterized by exhibiting coexistence of different equilibria \cite{GlansdorffPrigogine,NicolisPrigogine,Pismen2006}.
This is the physical context, where the life develops. Inhomogeneous initial conditions 
caused by, e.g., inherent fluctuations of macroscopic systems, generate the emergence 
of equilibria in different parts of space, which are usually identified as spatial domains. 
These domains are separated by domain walls or interface between the equilibria.
A classic example of this phenomena is the magnetic domains and walls \cite{Hubert}.
Depending on the configuration of the magnetization these walls are usually denominated as Ising, Bloch, and Neel.
Likewise, similar walls have been observed in liquid crystals, when a liquid crystal film is subjected 
to magnetic or electric fields \cite{OswaldPieranski}. 
In particular, nematic liquid crystals with planar anchoring exhibit Ising walls \cite{Chevallard}.
Close to the reorientation instability of the molecules, Fr\'eedericksz transition, this system is well described by the Allen-Cahn equation. 
Besides, using a photosensitive wall,  it is possible to induce a molecular reorientation in  a thin liquid crystal film \cite{ResidoriRiera2001}.
This type of device is usually called as liquid crystal light valve (see \cite{Residori2005} and references therein).
Due to the inhomogeneous illumination generated by light on the liquid crystal layer, the dynamics of molecular reorientation is described by
\begin{equation}
\partial_t  u(x_1,x_2,t)=\epsilon^2 \Delta u+\mu(x_1,x_2)u -u^3+ a \epsilon x_1 f(x_1,x_2),
\label{Eq-LCLV}
\end{equation}
where $u(x_1,x_2,t)$ accounts for the average rotational amplitude of the molecules, $t$, $x_1$, and $x_2$, 
respectively, stand for time and the transverse coordinates of the liquid crystal layer,
${x_1}$ is the direction in which the molecules are anchored, $f(x_1,x_2)=-\frac{1}{2}\partial_{x_1} \mu(x_1,x_2)$, and non dimensional parameters $\epsilon, a$ are positive. The function 
\begin{equation}
\label{def mu1}
\mu(x_1,x_2)=\mu_0+I_o e^{-\frac{x_1^2+x_2^2}{w^2}}, 
\end{equation}
which accounts the forcing given by the external electric field and the 
effect of the illuminated photo-sensitive wall characterized by the light intensity $I_o>0$, is typically sign changing i.e $-I_0<\mu_0<0$. This last  condition describes  the situation when  the electrical voltage applied to the liquid crystal sample is less than the Fr\'eedericksz 
voltage. The level set  $\{\mu(x_1, x_2)=0\}$ separates two disjoint regions where $\mu$ is of constant sign. For any   $x\in \{\mu>0\}$ the potential 
\[
U(z,x)=-\mu(x)\frac{z^2}{2}+\frac{z^4}{4}
\] 
has precisely two non degenerate minima of equal depth at  $z=\pm \sqrt{\mu(x)}$, while in the region $\{\mu<0\}$, $U$ is nonnegative and its only minimum occurs at $z=0$.  Motivated by this we will call the set $\{\mu<0\}\subset \R^2$ the bistable region and the set $\{\mu>0\}\subset \R^2$ the monostable region. Note that with the choice of the function $\mu$  in (\ref{def mu1}) the bistable region is a disc and the monostable region is its complement in $\R^2$. The objective of this paper is to understand how the location of the domain walls defined  as the set of  zeros of the solutions of (\ref{Eq-LCLV}) change when the parameters $\epsilon$ and $\alpha$ vary. For this purpose we will restrict our attention to the time independent solutions, the idea being that the system quickly relaxes to its stationary state.  

If one ignores the dependence on the transversal $x_2$ coordinate, the system exhibits two type of walls that 
separate domains that evanesce asymptotically \cite{Barboza2016,panayotis_1}.  One corresponds to the extension 
of Ising wall, standard kink, in this inhomogeneous system, which is a symmetric solution and centered 
in the region of the maximal  illumination i.e.  $x=0$ (since $\mu(x)$ attains its maximum in the origin). The other corresponds to a wall centered in the non-illuminated part,
shadow kink  \cite{Barboza2016,panayotis_1}. To understand the latter one can expand the solution around the point where $\mu(x)=0$. In this limit the  profile of the transition is described by  the second Painlev\'e equation \cite{panayotis_1,troy,2005math.ph...8062C}. This paper is devoted to understanding the physically relevant situation when the dependence on the second coordinate is not neglected and we take $t\to \infty$. In this limit the stationary solutions of (\ref{Eq-LCLV}) can be  characterized as the minima of the following  energy functional
\begin{equation}
\label{funct 0}
E(u)=\int_{\R^2}\frac{\epsilon}{2}|\nabla u|^2-\frac{1}{2\epsilon}\mu(x)u^2+\frac{1}{4\epsilon}u^4-a f_1(x) u,
\end{equation}
where $u\in  H^1(\R^2)$ and $\epsilon>0$, $a\geq 0$ are real parameters. 
More generally as in (\ref{def mu1}) we suppose that $\mu \in C^\infty(\R^2)$ is radial i.e. $\mu(x)=\mu_{\mathrm{rad}}(|x|)$, with $\mu_{\mathrm{rad}} \in C^\infty(\R)$ an even function. 
We take  $f=(f_1,f_2)\in C^\infty(\R^2,\R^2)$  also to be  radial i.e. $f(x)= f_{\mathrm{rad}}(|x|)\frac{x}{|x|}$, with $f_{\mathrm{rad}} \in C^\infty(\R)$ an odd function.
In addition we assume that 
\begin{equation}\label{hyp2}
 \begin{cases}
   \mu \in L^\infty(\R^2),\ \mu_{\mathrm{rad}}'<0 \text{ in }(0,\infty), \text{ and $ \mu_{\mathrm{rad}}(\rho)=0$ for a unique $\rho>0$},\medskip \\
  \text{$f \in L^1(\R^2,\R^2)\cap L^\infty(\R^2,\R^2)$, and $ f_{\mathrm{rad}}>0$ on $(0,\infty)$}.
 \end{cases}
\end{equation}
%
%
%In the physical context described in \cite{Barboza2016} the functions $\mu$ and $f=(f_1,f_2)$ are specific:
%\[
%\mu(x)=e^{\,-|x|^2}-\chi, \qquad \mbox{with some}\ \chi\in (0,1), \qquad f(x)=-\frac{1}{2}\nabla \mu(x).
%\]

The Euler-Lagrange equation of $E$ is
\begin{equation}\label{ode}
\ve^2 \Delta u+\mu(x) u-u^3+\ve a f_1(x)=0,\qquad x=(x_1,x_2)\in \R^2,
\end{equation} 
and we also write its weak formulation:
\begin{equation}\label{euler}
\int_{\R^2} -\epsilon^2 \nabla u\cdot \nabla \psi+\mu u \psi-u^3 \psi+\epsilon a f_1\psi=0,\qquad  \forall \psi \in H^1(\R^2),
\end{equation}
where $\cdot$ denotes the inner product in $\R^2$.
Note that due to the radial symmetry of $\mu$ and $f$, the energy \eqref{funct 0} and equation 
\eqref{ode} are invariant under the transformations $u(x_1,x_2)\mapsto -u(-x_1,x_2)$, and $u(x_1,x_2)\mapsto u(x_1,-x_2)$.

Our purpose in this paper is to study qualitative properties of the global minimizers of $E$ as the parameters $a$ and $\epsilon$ vary. In general we will assume that $\epsilon>0$ is small and $a\geq 0$ is fixed. In our previous work \cite{panayotis_1} and \cite{panayotis_2}, we examined respectively the cases of minimizers $v:\R\to\R$, and $v:\R^2\to\R^2$. In the present paper we follow the approach presented therein, and introduce several new ideas to address the specific issues occuring for minimizers $v:\R^2\to\R$. In particular, new variational arguments to determine the limit points of the zero level set of $v$, which is now a curve (cf. the conclusion of the proofs of Theorem \ref{thcv} (ii) and (iii)), and a computation of the energy that reduces to a one dimensional problem, by using iterated integrals.   

Proceeding as in \cite{panayotis_2}, one can see that under the above assumptions there exists a global minimizer $v$ of $E$ in $H^1(\R^2)$, namely 
that $E(v)=\min_{H^1(\R^2)} E$. 
In addition, we show that $v$ is a classical solution of \eqref{ode}, 
and $v$ is even with respect to $x_2$ i.e. $v(x_1,x_2)=v(x_1,-x_2)$. 
In the sequel, we will always denote by $v$ the global minimizer, and by $u$ an arbitrary critical point of $E$ in $H^1(\R)$.
Some basic properties are stated in: 
\begin{theorem}\label{theorem 1}
For $\epsilon\ll 1$, and $a\geq 0$ bounded (possibly dependent on $\epsilon$), 
let $v_{\epsilon, a}$ be a global minimizer of $E$, let $\rho>0$ be the zero of $\mu_{\mathrm{rad}}$ and let $\mu_1:=\mu_{\mathrm{rad}}'(\rho)<0$. The following  statements hold:
\begin{itemize}
\item[(i)] Let $\Omega\subset D(0;\rho)$ be an open set such that $v_{\epsilon, a}> 0$ (resp. $v_{\epsilon, a}< 0$) on $\Omega$, for every $\epsilon\ll 1$.
Then $v_{\epsilon, a}\to \sqrt{\mu}$ (resp. $v_{\epsilon, a}\to -\sqrt{\mu}$) in $C^0_{\mathrm{loc}}(\Omega)$. 
\item[(ii)] For every $\xi=\rho e^{i\theta}$, we consider the local coordinates $s=(s_1,s_2)$ in the basis $(e^{i\theta},i e^{i\theta})$, and the 
rescaled minimizers:
\[
w_{\epsilon,a}(s)= 2^{-1/2}(-\mu_1\ve)^{-1/3} v_{\epsilon,a}\Big( \xi+\ve^{2/3} \frac{s}{(-\mu_1)^{1/3}}\Big).
\]
Assuming that $\lim_{\epsilon\to 0}a(\epsilon)=a_0$, then as $\ve\to 0$, the function $w_{\epsilon, a}$ converges in $C^2_{\mathrm{loc}}(\R^2)$ up to subsequence, 
to a function $y$ bounded in $[s_0,\infty)\times \R$ for every $s_0\in\R$, which is a minimal  
solution of
\begin{equation}\label{pain}
\Delta y(s)-s_1 y(s)-2y^3(s)-\alpha=0, \qquad \forall s=(s_1,s_2)\in \R^2,
\end{equation}
with $\alpha=\frac{a_0 f_1(\xi)}{\sqrt{2}\mu_1}$. 
\item[(iii)] Assuming that $\lim_{\epsilon\to 0}a(\epsilon)= a_0$, then we have $\lim_{\epsilon\to 0}\frac{u_{\epsilon,a}(x)}{\epsilon}=
-\frac{a_0}{\mu(x)}f_1(x)$ uniformly on compact subsets of $\{|x|>\rho\}$.
\end{itemize}
\end{theorem}
Looking at the energy $E$ it is evident that as $\epsilon\to 0$ the modulus of the global minimizer $|v_{\epsilon,a}|$ should approach a nonnegative 
root of the polynomial
\[
-\mu(x)u+u^3-a\epsilon f_1(x)=0,
\] 
or in other words, $|v_{\epsilon, a}|\to \sqrt{\mu^+}$ as $\epsilon\to 0$ in some, perhaps weak, sense.  We observe for instance  that as a  corollary 
of Theorem \ref{theorem 1} (i) and Theorem \ref{thcv} (ii) below we obtain when $a<a_*$ the convergence in 
$C^0_{\mathrm{loc}}(D(0; \rho))$ (actually the uniform convergence holds in the whole plane). Because of the analogy between the functional $E$  and the Gross-Pitaevskii functional in theory of Bose-Einstein condensates we will call $\sqrt{\mu^+}$ the Thomas-Fermi limit of the global minimizer. 
Theorem \ref{theorem 1} gives account on how non smoothness of  the limit  of $v_{\epsilon, a}$  is mediated near the circumference $|x|=\rho$, where 
$\mu$ changes sign, through the solution of (\ref{pain}).

This  equation is a natural generalization of the second Painlev\'e ODE 
\begin{equation}
\label{pain 1d}
y''-sy-2y^3-\alpha=0, \qquad s\in \R.
\end{equation}
In \cite{panayotis_1} we showed that this last equation plays an analogous role in the one dimensional, scalar version of the energy $E$: 
\[
E(u,\R)=\int_{\R}\frac{\epsilon}{2}|u_x|^2-\frac{1}{2\epsilon}\mu(x)u^2+\frac{1}{4\epsilon}|u|^4-a f(x)u
\] 
where $\mu$ and $f$ are scalar functions satisfying similar hypothesis to those we have described above.  In this case the Thomas-Fermi limit of the global minimizer is simply $\sqrt{\mu^+(x)}$, which is non differentiable at  the points $x=\pm \xi$ which are the zeros of the even function $\mu$. Near these two points a rescaled version of the global minimizer approaches a solution of (\ref{pain 1d}) similarly as it is described in Theorem  \ref{theorem 1} (ii).

It is very important to realize that not every solution of (\ref{pain 1d}) can  serve as the limit of the global minimizer, since in our case the limiting solutions of (\ref{pain}) are necessarily minimal as well.
To explain what this means, let 
\[
E_{\mathrm{P_{II}}}(u, A)=\int_A \left[ \frac{1}{2} |\nabla u|^2 +\frac{1}{2}  s_1 u^2 +\frac{1}{2} u^4+\alpha u\right].
\]
By definition  a  solution of (\ref{pain}) is minimal if
\begin{equation}\label{minnn}
E_{\mathrm{P_{II}}}(y, \mathrm{supp}\, \phi)\leq E_{\mathrm{P_{II}}}(y+\phi, \mathrm{supp}\, \phi)
\end{equation}
for all $\phi\in C^\infty_0(\R^2)$.  This notion of minimality is standard for many problems in which the energy of a localized solution is actually infinite due to non compactness of the domain.

The study of minimal solutions of \eqref{pain 1d} was recently  initiated in \cite{panayotis_1} where we showed that the Hastings-McLeod solutions $h$ and $-h$, are the only minimal solutions of the homogeneous equation
\begin{equation}\label{phom} 
y''-sy-2y^3=0,  \, s\in \R, 
\end{equation}
which are bounded at $+\infty$.
We recall (cf. \cite{MR555581}) that $h:\R\to\R$ is positive, strictly decreasing ($h' <0$) and such that 
\begin{align}\label{asy0}
h(s)&\sim \mathop{Ai}(s), \qquad s\to \infty, \nonumber \\
h(s)&\sim \sqrt{|s|/2}, \qquad s\to -\infty.
\end{align}

On the other hand in \cite{panayotis_3} we considered when $a=0$, the odd minimizer $u$ of \eqref{funct 0}\footnote{Due to the symmetry of $\mu$ and $f$, $u$ is also a critical point of \eqref{funct 0} (cf. \cite{palais}), thus it solves \eqref{ode}.} in the class
$H^1_{\mathrm{odd}}(\R^2):=\{u \in H^1(\R^2): u(x_1,x_2)=-u(-x_1,x_2)\}$ of odd functions with respect to $x_1$, and following Theorem \ref{theorem 1} (ii), we established the existence of a nontrivial solution $y$ of the homogeneous equation \eqref{pain}. It has a form of a quadruple connection between the Airy function 
$\mathop{Ai}(x)$, the two one dimensional Hastings-McLeod solutions $\pm h(x)$ and the heteroclinic orbit $\eta(x)= \tanh(x/\sqrt{2})$ of the ODE $\eta''=\eta^3-\eta$.
Although we know (cf. \cite[Theorem 2.1]{panayotis_3}) that Theorem \ref{theorem 1} (ii) applied to the global minimizer $v$ in the homogeneous case $a=0$, gives at the limit either $y(s_1,s_2)=h(s_1)$ or  $y(s_1,s_2)=-h(s_1)$, we are
not aware if in the nonhomogeneous case $a\neq 0$, Theorem \ref{theorem 1} (ii) produces a new kind of minimal solution. This goes beyond the scope of the present paper.

Finally, regarding Theorem \ref{theorem 1} (iii) 
we note that since the sign of the local limit of the rescaled global minimizer in $|x|>\rho$ is determined by the sign of $f_1$, one may expect that the zero level set of $v_{\epsilon, a}$ is a smooth curve (cf. Lemma \ref{smooth}) partitioning the plane. In Theorem \ref{thcv} we will determine the limit of this level set according to the value of $a$, and discuss the dependence of the global minimizer on $a$, when $\epsilon\ll 1$.

Before stating our second result we recall that the heteroclinic orbit $\eta(x)= \tanh(x/\sqrt{2})$ ($\eta:\R\to (-1,1)$) of the ODE $\eta''=\eta^3-\eta$, connecting the two minima $\pm 1$ of the potential $W(u)=\frac{1}{4}(1-u^2)^2$ ($W:\R\to [0,\infty)$) plays a crucial role in the study of minimal solutions of the Allen-Cahn equation
\begin{equation}\label{pdeac}
\Delta u=u^3-u, u:\R^n\to \R.
\end{equation}
Again, we say that $u$ is a minimal solution of (\ref{pdeac}) if
\[
E_{\mathrm{AC}}(u, \mathrm{supp}\, \phi)\leq E_{\mathrm{AC}}(u+\phi, \mathrm{supp}\, \phi),
\]
for all $\phi\in C^\infty_0(\R^2)$, where 
\[
E_{\mathrm{AC}}(u,\Omega):=\int_\Omega \frac{1}{2}|\nabla u|^2+\frac{1}{4}(1-u^2)^2
\]
is the Allen-Cahn energy associated to \eqref{pdeac}. It is known \cite{savin} that in dimension $n\leq 7$, any minimal solution $u$ of (\ref{pdeac}) is either trivial i.e. $u\equiv\pm1$ or one dimensional i.e. $u(x)= \eta(  (x-x_0)\cdot \nu)$, for some $x_0\in\R^n$, and some unit vector $\nu\in\R^n$.

\begin{theorem}\label{thcv}
Let $Z=\{l\in \R^2 \text{ is a limit point of the set of zeros of $v_{\epsilon,a}$ as $\epsilon\to 0$}\}$.  
The following statements hold.
\begin{itemize}
\item[(i)]
When $a=0$ the global minimizer $v$ is unique up to change of $v$ by $-v$. It can be written as $v(x)=v_{\mathrm{rad}}(|x|)$, 
with $v_{\mathrm{rad}} \in C^\infty(\R)$, positive and even. 
\item[(ii)]
There exists a constant $a_*>0$ such that
for all $a\in (0, a_*)$, we have  up to change of $v(x_1,x_2)$ by $-v(-x_1,x_2)$:
$$ \{x_1<0, |x|=\rho\} \cup \{x_1=0, |x_2|\geq \rho\}\subset Z \subset \{|x|=\rho\} \cup \{x_1=0, |x_2|\geq \rho\},$$
and
\begin{equation}\label{cvn2}
\lim_{\epsilon\to 0} v(x+s\epsilon)=\sqrt{\mu^+(x)}, \forall x \in \R^2,
\end{equation} 
in  the $C^2_{\mathrm{ loc}}(\R)$ sense. The above asymptotic formula holds as well when $a=0$.
\item[(iii)]
Suppose that $f'_{\mathrm{rad}}(0)\neq 0$, then there exists a constant 
$a^*\geq a_*$ such that for all $a>a^*$ we have $Z=\{x_1=0\}$, and the global minimizer $v$ satisfies
\begin{equation}\label{cvn1}
\begin{aligned}
%&\lim_{\ve\to 0} v(\bar x+\ve s)=\sqrt{\mu(0)}\tanh(s_1\sqrt{\mu(0)/2}), \medskip\\
&\lim_{\ve\to 0} v( x+\ve s)=\begin{cases}\sqrt{\mu^+(x)}  &\text{for } x_1>0, \\
-\sqrt{\mu^+(x)}  &\text{for } x_1<0,
\end{cases}
\end{aligned}
\end{equation}
in the $C^2_{\mathrm{loc}}(\R)$ sense. Next, if $\bar x_{\epsilon,a}=(\bar t_{\epsilon,a},x_2)$ is a zero of $v_{\epsilon,a}$ with fixed ordinate $x_2$, then up to subsequence and  for a.e. $x_2\in(-\rho,\rho)$ we have
\begin{equation}\label{cvn00}
\lim_{\ve\to 0} v(\bar x+\ve s)=\sqrt{\mu(0,x_2)}\tanh(s_1\sqrt{\mu(0,x_2)/2}), \text{ in the $C^2_{\mathrm{loc}}(\R)$ sense}.
\end{equation}
Finally, when $f=-\frac{1}{2}\nabla \mu$ we have $a_*=a^*=\sqrt{2}$.  
\end{itemize}
\end{theorem}

Perhaps the most interesting and unexpected is the statement (ii) of the above theorem. It says that, at least in the limit $\epsilon\to 0$ the domain wall $Z$ is located at the border between the monostable region $\{\mu<0\}$ and the bistable region $\{\mu>0\}$. Physically this means that as the intensity of the illumination, measured by $a$, is relatively small then no defect is visibly seen. For this reason and by analogy with \cite{panayotis_1, panayotis_2} we call it the shadow domain wall.  As $a$ increases the shadow domain wall penetrates the bistable region becoming the standard domain wall, as described in  (iii).  

It is natural to expect in Theorem \ref{thcv} (ii) that $Z= \{x_1<0, |x|=\rho\} \cup \{x_1=0, |x_2|\geq \rho\}$. However, the energy considerations presented in the proof of Theorem \ref{thcv} do not exclude the existence of a limit point of the zeros of $v$ in the half-circle $\{x_1>0, |x|=\rho\}$. Actually, the existence of such a limit point induces an infinitesimal variation of the total energy that makes it difficult to detect. For the same reason, the limit \eqref{cvn00} in Theorem \ref{thcv} (iii) holds only for a.e. $x_2\in(-\rho,\rho)$. We also point out that the assumption that $f$ is radial, is essential to prove the existence of the constants $a_*$ and $a^*$ (cf. Lemma \ref{astar}).

\section{General results for minimizers and solutions}
In this section we gather general results for minimizers and solutions that are valid for any values of the parameters $\epsilon> 0$ and $a\geq 0$. 
We first prove the existence of global minimizers.
\begin{lemma}\label{lem exist min}
For every $\epsilon> 0$ and $a\geq 0$, there exists $v \in H^1(\R^2)$ such that $E(v)=\min_{H^1(\R^2)} E$. 
As a consequence, $v$ is a $C^\infty$ classical solution of \eqref{ode}. Moreover 
$v(x)\to 0$ as $|x|\to \infty$, and $v(x_1,x_2)=v(x_1,-x_2)$.
\end{lemma}
\begin{proof}
We proceed as in \cite[Lemma 2.1]{panayotis_2} to establish that the global minimizer exists and is a smooth solution of \eqref{ode} converging to $0$ as 
$|x|\to \infty$.
It remains to show that $v(x_1,x_2)=v(x_1,-x_2)$. We first note that $E(v,\R\times [0,\infty))=E(v,\R\times (-\infty,0])$. Indeed, if we assume without loss of generality that $E(v,\R\times [0,\infty))<E(v,\R\times (-\infty,0])$, the function
\begin{equation}
\tilde v(x_1,x_2)=
\begin{cases}
v(x_1,x_2) &\text{when  } x_2\geq 0,\\
v(x_1,-x_2) &\text{when  } x_2\leq 0,
\end{cases}
\end{equation}
has strictly less energy than $v$, which is a contradiction. Thus, $E(v,\R\times [0,\infty))=E(v,\R\times (-\infty,0])$, and as a consequence the 
function $\tilde v$ is also a global minimizer and a solution. It follows by unique continuation \cite{sanada} that $\tilde v\equiv v$ .
\end{proof}

To study the limit of solutions as $\epsilon\to 0$, we need uniform bounds in the different regions considered in Theorem \ref{theorem 1}.
\begin{lemma}\label{s3}
For $\epsilon a$ belonging to a bounded interval, let $u_{\epsilon,a}$ be a solution of \eqref{ode} converging to $0$ as $|x|\to \infty$. Then, the solutions $u_{\epsilon,a}$ and the maps $\epsilon\nabla u_{\epsilon,a}$ are uniformly bounded. 
\end{lemma}
\begin{proof}
We drop the indexes and write $u:=u_{\epsilon,a}$.
Since $|f|$, $\mu$, and $\epsilon a$ are bounded, the roots of the cubic equation in the variable $u$
$$u^3-\mu(x)u-\epsilon a f_1(x)=0$$ belong to a bounded interval, for all values of $x$, $\epsilon$, $a$. 
If $u$ takes positive values, then it attains its maximum $0\leq \max_{\R^2}u=u(x_0)$, at a point $x_0\in\R^2$. 
In view of \eqref{ode}: $$0\geq\epsilon^2 \Delta u(x_0)=u^3(x_0)-\mu(x_0)u(x_0)-\epsilon a f_1(x_0),$$ thus it follows that 
$u(x_0)$ is uniformly bounded above. In the same way, we prove the uniform lower bound for $u$. 
The boundedness of $\epsilon\nabla u_{\epsilon,a}$ follows from \eqref{ode}, the uniform bound of $u_{\epsilon,a}$, and standard elliptic estimates.
\end{proof}

\begin{lemma}\label{l2}
For $\epsilon\ll 1$ and $a$ belonging to a bounded interval, let $u_{\epsilon,a}$ be a solution of \eqref{ode} converging to $0$ as $|x|\to\infty$. 
Then, there exist a constant  $K>0$ such that 
\begin{equation}\label{boundd}
|u_{\ve,a}(x)|\leq K(\sqrt{\max(\mu (x),0)}+\ve^{1/3}), \quad \forall x\in \R^2.
\end{equation}
As a consequence, if for every $\xi=\rho e^{i\theta}$ we consider the local coordinates $s=(s_1,s_2)$ in the basis $(e^{i\theta},i e^{i\theta})$, then the rescaled functions 
$\tilde u_{\epsilon,a}(s)=\frac{u_{\ve,a}(\xi+ s\epsilon^{2/3})}{\epsilon^{1/3}}$ are uniformly bounded on the half-planes $[s_0,\infty)\times\R$, $\forall s_0\in\R$. 
\end{lemma}

\begin{proof}
For the sake of simplicity we drop the indexes and write $u:=u_{\epsilon,a}$. Let us define the following constants
\begin{itemize}
\item $ M>0$ is the uniform bound of $|u_{\epsilon,a}|$ (cf. Lemma \ref{s3}),
\item $\lambda>0$ is such that $3 \mu_{\mathrm{rad}}(\rho-h)\leq 2\lambda h$, $\forall h \in [0,\rho]$,
\item $F:=\sup_{\R^2}| f_1|$,
\item $\kappa>0$ is such that $\kappa^3\geq 3aF$, and $\kappa^4\geq 6\lambda$.
\end{itemize}
Next, we construct the following comparison function
\begin{equation}\label{compchi}
\chi(x)=
\begin{cases}
\lambda\Big(\rho-|x|+\frac{\epsilon^{2/3}}{2}\Big)&\text{ for } |x|\leq\rho,\\
\frac{\lambda}{2\epsilon^{2/3}}(|x|-\rho-\epsilon^{2/3})^2&\text{ for }\rho\leq |x|\leq\rho+\epsilon^{2/3},\\
0&\text{ for } |x|\geq\rho+\epsilon^{2/3}.
\end{cases}
\end{equation}
One can check that $\chi\in C^1(\R^2\setminus\{0\})\cap H^1(\R^2)$ satisfies $\Delta \chi\leq \frac{2\lambda}{\epsilon^{2/3}}$ in $H^1(\R^2)$. Finally, we define the function
$\psi:=\frac{u^2}{2}-\chi-\kappa^2\epsilon^{2/3}$, and compute:
\begin{align}
\epsilon^2 \Delta \psi&=\epsilon^2 (|\nabla u|^2+ u\Delta u-\Delta \chi)\nonumber\\
&\geq-\mu u^2+u^4-\epsilon a f_1 u-\epsilon^2 \Delta \chi \nonumber\\
&\geq-\mu u^2+u^4-\epsilon a F| u|-2\epsilon^{4/3}\lambda.
\end{align}
Now, one can see that when $x\in \omega:=\{x\in\R^2: \psi(x)> 0\}$, we have $\frac{u^4}{3}- \mu u^2\geq 0$, since
$$x \in\omega\cap \overline{D(0;\rho)}\Rightarrow \frac{u^4}{3}\geq \frac{2\lambda}{3}\Big(\rho-|x|+\frac{\epsilon^{2/3}}{2}\Big)u^2\geq \mu u^2 .$$
On the open set $\omega$, we also have: $\frac{u^4}{3}\geq \frac{\kappa^4}{3}\epsilon^{4/3}\geq 2\epsilon^{4/3}\lambda$, 
and $\frac{u^4}{3}\geq \frac{\kappa^3}{3}\epsilon|u|\geq \epsilon aF|u|$. Thus $\Delta \psi \geq 0$ on $\omega$ in the $H^1$ sense. 
To conclude, we apply Kato's inequality that gives: $\Delta \psi^+ \geq 0$ on $\R^2$ in the $H^1$ sense. Since $\psi^+$ is subharmonic with compact
support, we obtain by the maximum principle that $\psi^+\equiv 0$ or equivalently $\psi \leq 0$ on $\R^2$. The statement of the lemma follows by adjusting the constant $K$.
\end{proof}
\begin{lemma}\label{s3gg}
Assume that $a$ is bounded and let $u_{\epsilon,a}$ be solutions of \eqref{ode} uniformly bounded. 
Then, the functions $\frac{u_{\epsilon,a}}{\epsilon}$ and the maps $\nabla u_{\epsilon,a}$ are uniformly bounded on the sets $\{x:\, |x|\geq \rho_1\}$ 
for every $\rho_1>\rho$. 
\end{lemma}

\begin{proof}
We consider the sets $S:= \{ x:\  |x|\geq\rho_1\}\subset S':= \{ x:\  |x|>\rho'_1\}$, with $\rho<\rho'_1<\rho_1$, and define the constants:
\begin{itemize}
\item $ M>0$ which is the uniform bound of $|u_{\epsilon,a}|$,
\item $ \mu_0=-\mu_{\mathrm{rad}}(\rho'_1)>0$, 
\item $ f_\infty=\|f_1\|_{L^\infty}$,
\item $a^*:=\sup a(\epsilon)$,
\item $k=\frac{2a_*f_\infty}{\mu_0}>0$.
\end{itemize}
Next we introduce the function $\psi(x)=\frac{1}{2}(u^2-k^2\epsilon^2)$ satisfying:
\begin{align*}
\epsilon^2 \Delta \psi=\epsilon^2 \Delta\frac{u^2}{2}&\geq u^4+\mu_0u^2-\epsilon a_*f_\infty |u|  \ , \forall x\in S',\\
&\geq \mu_0 \psi , \ \forall x \in S' \text{ such that }  \psi(x)\geq 0.
\end{align*}
By Kato's inequality we have $\epsilon^2\Delta \psi^+ \geq  \mu_0\psi^+$ on $S'$, in the $H^1$ sense, and utilizing a standard 
comparison argument, we deduce that $\psi^+(x)\leq M^2 e^{-\frac{c}{\epsilon}d(x,\partial S')}$, $\forall x \in S$, and $
\forall \epsilon\ll 1$, where $d$ stands for the Euclidean distance, and $c>0$ is a constant. It is clear that 
$$d(x,\partial S')>-\frac{\epsilon}{c}\ln\Big(\frac{k^2
\epsilon^2}{2 M^2}\Big)\Rightarrow M^2e^{-\frac{c}{\epsilon}d(x,\partial S')}<\frac{k^2\epsilon^2}{2}\Rightarrow u^2<2k^2\epsilon^2.$$
Therefore, there exists $\epsilon_0$ such that 
\begin{equation}\label{asd1}
\frac{|u_{\epsilon,a}(x)|}{\epsilon}\leq \sqrt{2} k,\ \forall \epsilon<\epsilon_0,\ \forall x\in S.
\end{equation}
The boundedness of $\nabla u_{\epsilon,a}$ follows from \eqref{ode}, the uniform bound \eqref{asd1}, and standard elliptic estimates.
\end{proof}

\section{Proof of Theorems \ref{theorem 1} and \ref{thcv}}\label{proofs}

\begin{proof}[Proof of Theorem \ref{theorem 1} (i)] 
Without loss of generality we assume that $v_{\epsilon, a}> 0$ on $\Omega$.
Suppose by contradiction that $v$ does not converge uniformly to $\sqrt{\mu}$ on a closed set $F\subset \Omega$. Then there exist a sequence $\epsilon_n\to 0$
and a sequence $\{x_n\}\subset F$ such that
\begin{equation}\label{either}
\text{$|v_{\epsilon_n}(x_n)-\sqrt{\mu(x_n)}|\geq\delta$, for some $\delta>0$.}
\end{equation}
In addition, we may assume that up to a subsequence $\lim_{n\to\infty}x_n=x_0\in F$.
Next, we consider the rescaled functions
$\tilde v_n(s)=v_{\epsilon_n}(x_n+\epsilon_n s)$ that satisfy
\begin{equation}\label{asd2ccee}
\Delta \tilde v_n(s)+\mu( x_n+\epsilon_n s)\tilde v_n(s)-\tilde v^3_n(s)+\epsilon_n af_1( x_n+\epsilon_n s)=0 , \ \forall s \in \R^2.
\end{equation}
In view of the Lemma \ref{s3} and \eqref{asd2ccee}, $\tilde v_{n}$ and its first derivatives are uniformly bounded for $\epsilon \ll 1$. 
Moreover, by differentiating \eqref{asd2ccee}, one also obtains the boundedness of the second derivatives of $\tilde v_n$ on compact sets.
Thus, we can apply the theorem of Ascoli via a diagonal argument, and show that for a subsequence still called $\tilde v_n$,
$\tilde v_n$ converges in $C^2_{\mathrm{ loc}}(\R^2)$ to a function $\tilde V$, that we are now going to determine.
For this purpose, we introduce the rescaled energy
\begin{equation}
\label{functresee}
\tilde E(\tilde u)=\int_{\R^2}\Big(\frac{1}{2}|\nabla \tilde u(s)|^2-\frac{1}{2}\mu( x_n+\epsilon_n s)\tilde u^2(s)+\frac{1}{4}
\tilde u^4(s)-\epsilon_n a f_1( x_n+\epsilon_n s)\tilde u(s)\Big)\dd s=\frac{1}{\epsilon_n}E(u),\nonumber
\end{equation}
where we have set $\tilde u(s)=u_{\epsilon_n}( x_n+\epsilon_ns)$ i.e. $u_{\epsilon_n}(x)=\tilde u\big(\frac{x-x_n}{\epsilon_n}\big)$.
Let $\tilde \xi$ be a test function with support in the compact set $K$. We have $\tilde E(\tilde v_n+\tilde \xi,K)\geq \tilde E(\tilde v_n,K)$, and at the limit
$G_{0}( \tilde V+\tilde\xi,K)\geq  G_{0}(\tilde V,K)$, where $$ G_{0}(\psi,K)=\int_{K}\left[\frac{1}{2}|\nabla\psi|^2-\frac{1}{2}\mu(x_0)\psi^2
+\frac{1}{4}\psi^4\right],$$
or equivalently $G( \tilde V+\tilde\xi,K)\geq G(\tilde V,K)$, where
\begin{equation}\label{glee}
G(\psi,K)=\int_{K}\left[\frac{1}{2}|\nabla\psi|^2-\frac{1}{2}\mu(x_0)\psi^2+\frac{1}{4}\psi^4+\frac{(\mu(x_0))^2}{4}\right]
=\int_{K}\left[\frac{1}{2}|\nabla\psi|^2+\frac{1}{4}(\psi^2-\mu(x_0))^2\right].
\end{equation}
Thus, we deduce that $\tilde V$ is a bounded minimal solution of the P.D.E. associated to the functional \eqref{glee}:
\begin{equation}\label{odegl}
\Delta \tilde V(s)+(\mu(x_0)-\tilde V^2(s))\tilde V(s)=0.
\end{equation}
If $\tilde V$ is the constant solution $\sqrt{\mu(x_0)}$, then we have $\lim_{n\to\infty}v_{\epsilon_n}(x_n)=\sqrt{\mu(x_0)}$ which is excluded by \eqref{either}.
Therefore we obtain $\tilde V(s)=\sqrt{\mu(x_0)}\tanh(\sqrt{\mu(x_0)/2}(s-s_0)\cdot \nu)$, for some unit vector $\nu \in \R^2$, and some $s_0\in \R^2$.
This implies that $v_n $ takes negative values in the open disc $D(x_n; 2\epsilon_n|s_0|)$ for $\epsilon_n\ll 1 $, 
which contradicts the fact that $v_\epsilon>0$ on $\Omega$ for $\epsilon\ll 1$.
\end{proof}
\begin{proof}[Proof Theorem \ref{theorem 1} (ii)] 
For every $\xi=\rho e^{i\theta}$ we consider the local coordinates $s=(s_1,s_2)$ in the basis $(e^{i\theta},i e^{i\theta})$,
and we rescale the global minimizer $v$ by setting $\tilde v_{\epsilon,a}(s)=\frac{v_{\epsilon,a}(\xi+ s\epsilon^{2/3})}{\epsilon^{1/3}}$.
Clearly $\Delta \tilde v(s)=\epsilon \Delta v(\xi+s\epsilon^{2/3})$, thus,
\begin{equation}\label{oderes1}
\Delta \tilde v(s)+\frac{\mu(\xi+s\epsilon^{2/3})}{\epsilon^{2/3}} \tilde v(s)-\tilde v^3(s)+ a f_1(\xi+s\epsilon^{2/3})=0, \qquad \forall s\in \R^2.\nonumber
\end{equation}
Writing $\mu(\xi+h)=\mu_1 h_1+h\cdot A(h)$, with $\mu_1:=\mu'_{\mathrm{rad}}(\rho)<0$, $A \in C^\infty(\R^2,\R^2)$, and $A(0)=0$, we obtain
\begin{equation}\label{oderes2}
\Delta \tilde v(s)+(\mu_1 s_1 + A(s \epsilon^{2/3})\cdot s) \tilde v(s)-\tilde v^3(s)+ a f_1(\xi+s\epsilon^{2/3})=0,\qquad  \forall s\in \R^2.
\end{equation}
Next, we define the rescaled energy by
\begin{equation}
\label{functres2}
\tilde E(\tilde u)=\int_{\R^2}\Big(\frac{1}{2}|\nabla\tilde u(s)|^2-\frac{\mu(\xi+s \epsilon^{2/3})}{2\epsilon^{2/3}}\tilde u^2(s)+\frac{1}{4}\tilde u^4(s)- a f_1(\xi+s \epsilon^{2/3}) \tilde u(s)\Big)\dd s.
\end{equation}
With this definition $\tilde E(\tilde u)=\frac{1}{\epsilon^{5/3}}E(u)$.
From Lemma \ref{l2} and \eqref{oderes2}, it follows that $\Delta \tilde v$, and also $\nabla\tilde v$, are uniformly bounded on compact sets. 
Moreover, by differentiating \eqref{oderes2} we also obtain the boundedness of the second derivatives of $\tilde v$.
Thanks to these uniform bounds, we can apply the theorem of Ascoli via a diagonal argument to obtain the convergence of $\tilde v$ in 
$C^2_{\mathrm{ loc}}(\R^2)$ (up to a subsequence) 
to a solution $\tilde V $ of the P.D.E. 
\begin{equation}\label{oderes4}
\Delta \tilde V(s)+\mu_1 s_1 \tilde V(s)-\tilde V^3(s)+ a_0 f_1(\xi)=0, \ \forall s\in \R^2, \text{ with } a_0:=\lim_{\epsilon\to 0}a(\epsilon),
\end{equation}
which is associated to the functional
\begin{equation}
\label{functres4}
\tilde E_0(\phi,J)=\int_{J}\Big(\frac{1}{2}|\nabla\phi(s)|^2-\frac{\mu_1}{2} s_1 \phi^2(s)+\frac{1}{4}\phi^4(s)- a_0 f_1(\xi)\phi(s) \Big)\dd s.
\end{equation}
Setting $y(s):=\frac{1}{\sqrt{2}(-\mu_1)^{1/3}}\tilde V\big(\frac{s}{(-\mu_1)^{1/3}}\big)$, \eqref{oderes4} reduces to \eqref{pain}, 
that is, $y$ solves \eqref{pain} with $\alpha=\frac{a_0f_1(\xi)}{\sqrt{2}\mu_1}$.
Finally, we can see as in the previous proof that
the limit $\tilde V$ obtained in \eqref{oderes4} as well as the solution $y$ of \eqref{pain} are 
minimal in the sense of definition \eqref{minnn}.
\end{proof}

\begin{proof}[Theorem \ref{theorem 1} (iii)] 
For every $x_0\in\R^2$ such that $|x_0|>\rho$, we consider the rescaled minimizers
$\tilde v_{\epsilon,a}(s)= \frac{v_{\epsilon,a}(x_0+\epsilon s)}{\epsilon}$, with $s=(s_1,s_2)$, satisfying 
\begin{equation}\label{asd2}
\Delta \tilde v(s)+\mu(x_0+\epsilon s)\tilde v(s)-\epsilon^2\tilde v(s)^3+af_1(x_0+\epsilon s)=0 , \ \forall s \in \R^2.
\end{equation}
In view of the bound provided by Lemma \ref{s3gg} and \eqref{asd2}, we can see that the first derivatives of 
$\tilde v_{\epsilon,a}$ are uniformly bounded on compact sets for $\epsilon \ll 1$.
Moreover, by differentiating \eqref{asd2}, one can also obtain the boundedness of the second derivatives of $\tilde v$ on compact sets.
As a consequence, we conclude that $\lim_{\epsilon\to 0, a\to a_0}\tilde v_{\epsilon,a}(s)=\tilde V(s)$ in $C^2_{\mathrm{loc}}$, 
where $\tilde V(s)\equiv-\frac{a_0}{\mu(x_0)}f_1(x_0 )$ is the unique bounded solution of
\begin{equation}\label{asd3}
\Delta \tilde V(s)+\mu(x_0)\tilde V(s)+a_0f_1(x_0)=0 , \ \forall s \in \R^2.
\end{equation}
Indeed, consider a smooth and bounded solution $\phi:\R^2\to\R$ of $\Delta \phi=W'(\phi)$ where the potential $W:\R\to \R$ is smooth 
and strictly convex. Then, we have $\Delta(W(\phi))=|W'(\phi)|^2+W''(\phi)|\nabla \phi|^2\geq 0$, and since $W(\phi)$ is bounded we deduce 
that $W(\phi)$ is constant. Therefore, $\phi\equiv \phi_0$ where $\phi_0\in \R$ is such that $W'(\phi_0)=0$. 
To prove the uniform convergence $\frac{v_{\epsilon,a}(x)}{\epsilon}\to-\frac{a_0}{\mu(x)}f_1(x)$ on compact subsets of $\{|x|>\rho\}$, we proceed by contradiction.
Assuming that the uniform convergence does not hold, one can find a sequence $\epsilon_n\to 0$, 
a sequence $a_n\to a_0$, and a sequence $x_n \to x_0$, with $|x_0|>\rho$, such that
$\Big|\frac{v_{\epsilon_n,a_n}(x_n)}{\epsilon_n}+\frac{a_0}{\mu(x_n)}f_1(x_n)\Big|\geq \delta$, for some $\delta>0$. 
However, by reproducing the previous arguments,
it follows that the rescaled functions $\tilde v_{n}(s)= \frac{v_{\epsilon_n,a_n}(x_n+\epsilon_n s)}{\epsilon_n}$ converge 
in $C^2_{\mathrm{loc}}$ to the constant $\tilde V(s)\equiv-\frac{a_0}{\mu(x_0)}f_1(x_0)$. Thus, we have reached a contradiction.
\end{proof}

\begin{proof}[Proof of Theorem \ref{thcv} (i)] 
We first notice that $v \not\equiv 0$ for $\epsilon \ll 1$. 
Indeed, by choosing a test function $\psi\not\equiv 0$ supported in $D(0;\rho)$, and such that $\psi^2<2\mu$, one can see that 
\[
E(\psi)=\frac{\epsilon}{2}\int_{\R^2}|\nabla \psi|^2+\frac{1}{4\epsilon}\int_{\R^2}\psi^2(\psi^2-2\mu)<0, \qquad  \epsilon\ll 1.
\] 
Let $x_0\in \R^2$ be such that $v(x_0)\neq 0$. Without loss of generality we may assume that $v(x_0)>0$. Next, consider $\tilde v=|v|$ which is another global minimizer and thus another solution.
Clearly, in a neighborhood of $x_0$ we have $v=|v|$, and as a consequence of the unique continuation principle (cf. \cite{sanada})
we deduce that $v\equiv \tilde v\geq 0$ on $\R^2 $. Furthermore, 
the maximum principle implies that $v>0$, since $v\not\equiv 0$.
To prove that $v$ is radial we consider the reflection with respect to the line $x_1=0$.
We can check that $E(v,\{x_1>0\})=E(v,\{x_1<0\})$, since otherwise by even reflection we can construct a map in $H^1$ with energy smaller  than $v$. Thus, the map 
$\tilde v(x) = v(|x_1|,x_2)$ is also a minimizer, and since $\tilde v= v$ on $\{x_1>0\}$, it follows by unique continuation that 
$\tilde v\equiv v$ on $\R^2$. Repeating the same argument for any line of reflection, we deduce that $v$ is radial.
To complete the proof, it remains to show the uniqueness of $v$ up to change of $v$ by $-v$. Let $\tilde v$ be another global minimizer such that $\tilde v>0$, and $\tilde v\not\equiv v$. 
Choosing $\psi=u$ in \eqref{euler}, we find for any solution $u\in H^1(\R^2)$ of \eqref{ode} the following alternative expression of the energy:
\begin{equation}\label{enealt}
E(u)=-\int_{\R^2}  \frac{u^4}{4\epsilon}.
\end{equation}
Formula \eqref{enealt} implies that $v $ and $\tilde v$ intersect for $|x|=r>0$. However, setting
\begin{equation*}
w(x)=\begin{cases}
      v(x) &\text{ for } |x|\leq r\\
       \tilde v(x) &\text{ for } |x|\geq r,
     \end{cases}
 \end{equation*}
we can see that $w$ is another global minimizer, and again by the unique continuation principle we have $w\equiv v\equiv\tilde v$. 
\end{proof}

\begin{proof}[Proof of Theorem \ref{thcv} (ii), (iii)] 
We first establish two lemmas. 
\begin{lemma}\label{smooth}
Let $a>0$ and $\rho_0\in(0,\rho)$ be fixed, and set $l:=\frac{\sqrt{\mu_{\mathrm{rad}}(\rho_0)}}{2\mu(0)}$, $\lambda:=\frac{\sqrt{2} \tanh^{-1}(8/9)}{\sqrt{\mu_{\mathrm{rad}}(\rho_0)}}$, and $\lambda':=\frac{\mu_{\mathrm{rad}}(\rho_0) }{2\cosh^2(\lambda\sqrt{\mu(0)/2})}$.
Then, there exist $\epsilon_0>0$ such that
\begin{itemize}
\item[(i)] for every $\epsilon \in (0,\epsilon_0)$ 
the set $Z_{\epsilon}:=\{\bar x\in D(0;\rho_0): v_{\epsilon,a}(\bar x)=0\}$ is a smooth one 
dimensional manifold. Let $\nu(\bar x)$ %and $\tau(\bar x) $, 
be a unit normal vector %,and a unit tangent vector
at $\bar x\in Z_\epsilon$.
\item[(ii)] for every $\epsilon\in  (0,\epsilon_0)$, $\bar x\in Z_\epsilon$, and $|s|\leq l$, we have 
$|v(\bar x+\epsilon s)|\leq\frac{1}{2} \sqrt{\mu_{\mathrm{rad}}(\rho_0)}$.
\item[(iii)] for every $\epsilon \in (0,\epsilon_0)$, and $\bar x\in Z_\epsilon$, we have 
$|v(\bar x+\epsilon\lambda \nu)|\geq \frac{3}{4}\sqrt{\mu_{\mathrm{rad}}(\rho_0)}$,
\item[(iv)] for every $\epsilon \in (0,\epsilon_0)$, $\bar x\in Z_\epsilon$, and $t\in[-\lambda,\lambda]$ we have 
$\epsilon \big|\frac{\partial v}{\partial \nu}(\bar x+\epsilon t \nu)\big|\geq \lambda'$.
\end{itemize}
\end{lemma}
\begin{proof}
To prove (i) it is sufficient to establish that there exists $\epsilon_0>0$ such that for every $\epsilon \in (0,\epsilon_0)$ and 
$\bar x \in Z_\epsilon$, we have $\nabla v_{\epsilon,a}(\bar x)\neq 0$. Assuming by contradiction that this does not hold, 
we can find a sequence $\epsilon_n\to 0$, and a sequence
$ Z_{\epsilon_n}\ni \bar x_n \to x_0 \in \overline{D(0;\rho_0)}$ such that $\nabla v_{\epsilon_n,a}(\bar x_n)= 0$. 
However, by considering the rescaled functions 
$\tilde v_n(s)=v_{\epsilon_n,a}(\bar x_n+\epsilon_n s)$, it follows as in the proof of Theorem \ref{theorem 1} (i) that $\tilde v_n$ converges in $C^2_{\mathrm{ loc}}(\R^2)$ (up to a subsequence) to 
$\tilde V(s)=\sqrt{\mu(x_0)}\tanh(\sqrt{\mu(x_0)/2}(s\cdot \nu))$, where $\nu\in\R^2$ is a unit vector. Since $\nabla \tilde V(0)\neq 0$, we have reached a contradiction.
To prove (ii), we proceed again by contradiction, and assume that we can find a sequence $\epsilon_n\to 0$, a sequence
$ Z_{\epsilon_n}\ni \bar x_n \to x_0 \in \overline{D(0;\rho_0)}$, and a sequence $\overline{D(0;l)}\ni s_n\to s_0$ such that 
$|v(\bar x_n+\epsilon_n s_n)|> \sqrt{\mu_{\mathrm{rad}}(\rho_0)}/2$. As before, we obtain that $\tilde v_n(s)=v_{\epsilon_n,a}(\bar x_n+\epsilon_n s)$,
converges in $C^2_{\mathrm{ loc}}(\R^2)$ to 
$\tilde V(s)=\sqrt{\mu(x_0)}\tanh(\sqrt{\mu(x_0)/2}(s\cdot \nu))$. In particular, it follows that 
$$\lim_{n\to\infty}|v_{\epsilon_n,a}(\bar x_n+\epsilon_n s_n)|=\sqrt{\mu(x_0)}|\tanh(\sqrt{\mu(x_0)/2}(s_0\cdot \nu))|\leq \frac{\mu(0)l}{\sqrt{2}}
<\sqrt{\mu_{\mathrm{rad}}(\rho_0)}/2,$$
which is a contradiction. The proofs of (iii) and (iv) are similar.
%Proceeding as in Lemma \ref{smooth} (ii), and setting  we are going to show that 
%there exists $N$ such that
%To prove (a), we proceed by contradiction, and assume that we can find a subsequence $n_k\to \infty$, and a sequence
%$ \Gamma_{n_{k}}\ni \bar x_{k} \to x_0 \in \overline{D(0;\rho_0)}$, such that $|v_{n_k}(\bar x_{k}+\epsilon_{n_k}\lambda \nu_{n_k})|< \frac{3}{4}\sqrt{\mu_{\mathrm{rad}}(\rho_0)}$, and $\nu_{n_k}(\bar x_k)\to\nu_0$.
%Since $\tilde v_k(s)=v_{n_k}(\bar x_{k}+\epsilon_{n_k} s)$
%converges in $C^2_{\mathrm{ loc}}(\R^2)$ to 
%$\tilde V(s)=\sqrt{\mu(x_0)}\tanh(\sqrt{\mu(x_0)/2}(s\cdot \nu_0))$, it follows that 
%$$\lim_{k\to\infty}|v_{n_k}(\bar x_{k}+\epsilon_{n_k}\lambda \nu_{n_k})|=\sqrt{\mu(x_0)}\tanh(\lambda\sqrt{\mu(x_0)/2})\geq \frac{8}{9}\sqrt{\mu_{\mathrm{rad}}(\rho_0)},$$
%which is a contradiction. Similarly, if (b) does not hold, we can find a subsequence $n_k\to \infty$, a sequence
%$ \Gamma_{n_{k}}\ni \bar x_{k} \to x_0 \in \overline{D(0;\rho_0)}$, and a sequence $[-\lambda,\lambda]\ni t_k\to t_0$ such that $\epsilon_{n_k}\big|\frac{\partial v_{n_k}}{\partial \nu_{n_k}}(\bar x_{k}+\epsilon_{n_k} t_k \nu_{n_k})\big|<\lambda'$, and $%\nu_{n_k}(\bar x_k)\to\nu_0$.
%Since
%$$\lim_{k\to\infty}\epsilon_{n_k}\Big|\frac{\partial v_{n_k}}{\partial \nu_{n_k}}(\bar x_{k}+\epsilon_{n_k} t_k \nu_{n_k})\Big|=\Big|\frac{\partial \tilde V}{\partial \nu_{0}}( t_0 \nu_{n_0})\Big|\geq\sqrt{2}\lambda',$$ we reach again a contradiction.
\end{proof}

\begin{lemma}\label{astar}
Let
\begin{equation}\label{asbb}
 a_*:=\inf_{x_1\leq 0, |x|<\rho}\frac{\sqrt{2}(\mu(x))^{3/2}}{3\int_{-\sqrt{\rho^2-x_2^2}}^{x_1} |f_1(t,x_2)|\sqrt{\mu(t,x_2)} \dd t},
\end{equation}
and
\begin{equation}\label{ashh}
 a^*:=\sup_{x_1< 0, |x|\leq\rho} \frac{\sqrt{2}\big((\mu(0,x_2))^{3/2}-(\mu(x))^{3/2}\big)}{3\int_{x_1}^0|f_1(t,x_2)|\sqrt{\mu(t,x_2)}\dd t},
\end{equation}
then we have $a_*\in (0,\infty)$ and 
\begin{equation}\label{inegaaa}
 a_*\leq\frac{2\sqrt{2}\int_{|r|<\rho}(\mu_{\mathrm{rad}}(r))^{3/2}\dd r}{3\int_{D(0;\rho)}|f_1|\sqrt{\mu} }\leq a^*.
\end{equation}
Moreover, if $f'_{\mathrm{rad}}(0)> 0$, then $a^*<\infty$. 
Finally, if $f=-\frac{1}{2}\nabla\mu$, then $a_*=a^*=\sqrt{2}$.
\end{lemma}
\begin{proof}
We first check that $a_*\in (0,\infty)$ and $a^*\in [a_*,\infty]$. Let us define the auxilliary function
$$\{x\in\R^2: x_1\leq 0, |x|\leq\rho\}\ni x\to \beta_*(x)=\frac{\sqrt{2}}{3}(\mu(x))^{3/2}-a\int_{-\sqrt{\rho^2-x_2^2}}^{x_1} |f_1(t,x_2)|\sqrt{\mu(t,x_2)} \dd t,$$
and compute $\frac{\partial\beta_*}{\partial x_1}(x)=\big(\frac{\sqrt{2}}{2}\mu'_{\mathrm{rad}}(r)-af_{\mathrm{rad}}(r)\big) \sqrt{\mu(x)}\cos \theta$, where $x=(r\cos\theta,r\sin\theta)$. 
It is clear that for sufficiently small $a_1>0$ and $\gamma>0$, we have $\frac{\partial\beta_*}{\partial x_1}(x)> 0$ provided that $x_1<0$, 
$\rho-\gamma<|x|<\rho$, and $a\leq a_1$. Since $\beta_*(x)=0$ for $|x|=\rho$, it follows that 
$\beta_*(x)\geq 0$ provided that $x_1\leq 0$, $\rho-\gamma\leq|x|\leq\rho$, and $a\leq a_1$. 
There also exists $a_2>0$ such that for $a\leq a_2$, we have $\beta_* \geq 0$ on the set $\{x_1\leq 0, |x|\leq\rho-\gamma\}$. 
Thus, we can see that $a_*\geq\min(a_1,a_2)>0$. Furthermore, since the inequalities 
$a_*\leq\frac{2\sqrt{2}(\mu(0,x_2))^{3/2}}{3\int_{|t|<\sqrt{\rho^2-x_2^2}}|f_1(t,x_2)|\sqrt{\mu(t,x_2)}\dd t }\leq a^*$ hold for every 
$x_2\in (-\rho,\rho)$, we obtain after an integration \eqref{inegaaa}.
Next, we define a second auxilliary function
$$\{x\in\R^2: x_1\leq 0, |x|\leq\rho\}\ni x\to \beta^*(x)=\frac{\sqrt{2}}{3}[(\mu(0,x_2))^{3/2}-(\mu(x))^{3/2}]-a\int_{x_1}^{0} |f_1(t,x_2)|\sqrt{\mu(t,x_2)} \dd t,$$
and compute $\frac{\partial\beta^*}{\partial x_1}(x)=\big(\frac{\sqrt{2}}{2}\mu'_{\mathrm{rad}}(r)+af_{\mathrm{rad}}(r)\big) \sqrt{\mu(x)}|\cos \theta|$, where $x=(r\cos\theta,r\sin\theta)$. 
Since $f'_{\mathrm{rad}}(0)> 0$, one can see that $\frac{\sqrt{2}}{2}\mu''_{\mathrm{rad}}(r)+af'_{\mathrm{rad}}(r)> 0$,
provided that $r\in[0,\gamma]$ and $a\geq a_3$, with $\gamma>0$ sufficiently small, and $a_3>0$ sufficiently big. Thus, 
$\frac{\sqrt{2}}{2}\mu'_{\mathrm{rad}}(r)+af_{\mathrm{rad}}(r)\geq 0$, and $\frac{\partial\beta^*}{\partial x_1}(x)> 0$, when $r=|x|\leq\gamma$,
$x_1<0$, and $a\geq a_3$.
On the other hand it is clear that for sufficiently big $a_4>0$, we have $\frac{\partial\beta^*}{\partial x_1}(x)> 0$ provided that $x_1<0$, 
$\gamma\leq |x|<\rho$, and $a\geq a_4$. Since $\beta^*(x)=0$ for $x_1=0$, it follows that 
$\beta^*(x)\leq 0$ provided that $x_1\leq 0$, $|x|\leq\rho$, and $a\geq \max(a_3,a_4)$. This proves that $a^*\leq \max(a_3,a_4)$.
Finally, one can check that $a_*=a^*=\sqrt{2}$ when $f=-\frac{1}{2}\nabla\mu \Rightarrow f_1=-\frac{1}{2}\frac{\partial\mu}{\partial x_1}$, by 
computing the integrals appearing in the denominators of \eqref{asbb}, \eqref{ashh}.
\end{proof}

The minimum of the energy defined in \eqref{funct 0} is nonpositive and tends to $-\infty$ as $\epsilon \to 0$. Since we are interested in the behavior of the minimizers as $\epsilon \to 0$, it is useful to define a renormalized energy, which is obtained by adding to \eqref{funct 0} a suitable term so that the result is tightly bounded from above. We define the renormalized energy as 
\begin{equation}\label{renorm}
\mathcal{E}(u):=E(u)+\int_{|x|<\rho}\frac{\mu^2}{4\epsilon}=\int_{\R^2}\frac{\epsilon}{2}|\nabla u|^2+\int_{|x|<\rho}
\frac{(u^2-\mu)^2}{4\epsilon}+\int_{|x|>\rho}\frac{u^2(u^2-2\mu)}{4\epsilon}- a \int_{\R^2}f_1 u ,
\end{equation}  
and claim the bound:
\begin{lemma}\label{bbnm}
\begin{equation}\label{quaq1}
\limsup_{\epsilon\to 0}\mathcal{E}_{\epsilon,a}(v_{\epsilon,a})\leq\min\Big(0,\frac{2\sqrt{2}}{3}\int_{-\rho}^\rho (\mu_{\mathrm{rad}}(r))^{3/2}\dd r
-a\int_{D(0;\rho)}|f_1|\sqrt{\mu}\Big),
 \text{ for arbitrary fixed $a$}.
\end{equation}
\end{lemma}
\begin{proof}%[Proof of \eqref{quaq1}]
Let us consider the $C^1$ piecewise function:
\begin{equation}
\psi_\epsilon(x)=
\begin{cases}
\sqrt{\mu(x)} &\text{for } |x|\leq \rho-\epsilon^{2/3} \\
k_\epsilon \epsilon^{-1/3}(\rho-|x|)  &\text{for }\rho-\epsilon^{2/3} \leq|x|\leq \rho\\
0  &\text{for } |x|\geq \rho
\end{cases}, \nonumber
\end{equation} 
with $k_\ve$ defined by $k_\epsilon \epsilon^{1/3}=\sqrt{\mu_{\mathrm{rad}}(\rho- \epsilon^{2/3})}\Longrightarrow k_\epsilon=|\mu'_{\mathrm{rad}} (\rho)|^{\frac{1}{2}}+o(1)$.
Since $\psi \in H^1(\R^2)$, it is clear that $\mathcal E(v)\leq \mathcal E(\psi)$. 
We check that $\mathcal E(\psi)=\frac{\pi|\mu_1|\rho}{6}|\epsilon\ln\epsilon|+\mathcal O(\epsilon)$, since it is the sum of the following integrals:
%$$\int_{|x|>\rho}\frac{|\phi|^2(|\phi|^2-2\mu)}{4\epsilon}=\mathcal O(\epsilon),$$
$$\int_{\rho-\epsilon^{2/3}<|x|<\rho}\frac{(\psi^2-\mu)^2}{4\epsilon}=\mathcal O(\epsilon),\ \int_{|x|>\rho-\epsilon^{2/3}}\frac{\epsilon}{2} |\nabla \psi|^2=\mathcal O(\epsilon),$$
$$\int_{|x|\leq\rho-\epsilon^{2/3}}\frac{\epsilon}{2}\frac{ |\mu'_{\mathrm{rad}}(|x|)|^2}{4\mu}=\frac{\epsilon|\mu_1|}{8}
\int_{|x|\leq\rho-\epsilon^{2/3}}\frac{1}{\rho-|x|}+\mathcal O(\epsilon)=\frac{\pi|\mu_1|\rho}{6}|\epsilon\ln\epsilon|+\mathcal O(\epsilon).$$
Thus, $\limsup_{\epsilon\to 0}\mathcal{E}_{\epsilon,a}(v_{\epsilon,a})\leq \limsup_{\epsilon\to 0}\mathcal{E}_{\epsilon,a}(\psi_\epsilon)=0$.

Next, we set $\zeta_\epsilon:= \epsilon^{-\beta}$, with $\beta\in (\frac{1}{3},\frac{4}{9})$, and define the $C^1$ piecewise functions:
\begin{equation}
\l_\epsilon(x_2)=
\begin{cases}
\frac{\psi_\epsilon (\epsilon\zeta_\epsilon,x_2)}{\tanh \big( \zeta_\epsilon\frac{\psi_\epsilon(0,x_2)}{\sqrt{2}}\big)}
&\text{for } |x_2|\leq (\rho^2-\epsilon^2\zeta_\epsilon^2)^{1/2}
\\
0 &\text{for } |x_2|\geq (\rho^2-\epsilon^2\zeta_\epsilon^2)^{1/2}
\end{cases}, \nonumber
\end{equation} 
and
\begin{equation}
\chi_\epsilon(x)=
\begin{cases}
l_\epsilon (x_2)\tanh\big(\frac{x_1\psi_\epsilon(0,x_2)}{\sqrt{2}\epsilon}\big) &\text{for } |x_1|\leq \epsilon \zeta_\epsilon
\\
\psi_\epsilon (x)  &\text{for } x_1\geq \epsilon\zeta_\epsilon\\
-\psi_\epsilon (x)  &\text{for } x_1\leq -\epsilon\zeta_\epsilon
\end{cases}. \nonumber
\end{equation} 
We also consider the sets $$D_\epsilon^1:=\{(x_1,x_2): |x_1|\leq \epsilon \zeta_\epsilon,  |x_2|\leq \big((\rho^2-\epsilon^{2/3})^2-\epsilon^2\zeta_\epsilon^2\big)^{1/2}\},$$
$$D_\epsilon^2:=\{(x_1,x_2): |x_1|\leq \epsilon \zeta_\epsilon,  |x_2|\geq  \big((\rho^2-\epsilon^{2/3})^2-\epsilon^2\zeta_\epsilon^2\big)^{1/2}, |x|\leq \rho\},$$
and
$$D_\epsilon^3:=\{(x_1,x_2): |x_1|\geq \epsilon \zeta_\epsilon,  |x|\leq \rho\}.$$
One the one hand, it is clear that $$\lim_{\epsilon\to0}-a\int_{\R^2}f_1\chi_\epsilon=-a\int_{D(0;\rho)}|f_1|\sqrt{\mu},$$ 
and
$$\lim_{\epsilon\to0}\int_{D_\epsilon^3}\big(\frac{\epsilon}{2}|\nabla \chi_\epsilon|^2+\frac{(\chi_\epsilon^2-\mu)^2}{4\epsilon}\big)=0.$$
In addition, it is a simple calculation to verify that
$$\lim_{\epsilon\to0}\int_{D_\epsilon^2}\big(\frac{\epsilon}{2}|\nabla \chi_\epsilon|^2+\frac{(\chi_\epsilon^2-\mu)^2}{4\epsilon}\big)=0.$$
On the other hand when $|x_2|\leq \big((\rho^2-\epsilon^{2/3})^2-\epsilon^2\zeta_\epsilon^2\big)^{1/2}=:\tau_\epsilon$, we have
$l_\epsilon^2(x_2)=\mu(0,x_2)+\mathcal O(\epsilon^2\zeta_\epsilon^2)$, uniformly in $x_2$.
Our claim is that
\begin{equation}\label{quabb}
\lim_{\epsilon\to0}\int_{D_\epsilon^3}\big(\frac{\epsilon}{2}|\nabla \chi_\epsilon|^2+\frac{(\chi_\epsilon^2-\mu)^2}{4\epsilon}\big)=\frac{2\sqrt{2}}{3}\int_{-\rho}^\rho (\mu(0,x_2))^{3/2}\dd x_2.
\end{equation}
Indeed, setting $\tilde \chi(x_1,x_2)=\sqrt{\mu(0,x_2)}\tanh\Big(x_1\sqrt{\frac{\mu(0,x_2)}{2}}\Big)$, we can see that
$\int_{D_\epsilon^3}\big(\frac{\epsilon}{2}|\nabla \chi_\epsilon|^2+\frac{(\chi_\epsilon^2-\mu)^2}{4\epsilon}\big)$ is the sum of the following integrals:
$$\int_{D_\epsilon^1}\frac{\mu^2}{4\epsilon}=\int_{D_1^\epsilon}\frac{\mu^2(0,x_2)}{4\epsilon}+\mathcal O(\zeta_\epsilon^2\epsilon), $$
$$\int_{D^1_\epsilon}\frac{\epsilon}{2}\Big|\frac{\partial\chi_\epsilon}{\partial x_1}\Big|^2=\int_{|x_1|<\zeta_\epsilon, |x_2|<\tau_\epsilon}\frac{1}{2}\frac{l_\epsilon^2(x_2)}{\mu(0,x_2)}\Big|\frac{\partial\tilde \chi}{\partial x_1}\Big|^2 
=\int_{|x_1|<\zeta_\epsilon, |x_2|<\tau_\epsilon}\frac{1}{2}\Big|\frac{\partial\tilde \chi}{\partial x_1}\Big|^2 +\mathcal O( \epsilon^{4/3}\zeta_\epsilon^2),$$ 
$$\int_{D^1_\epsilon}\frac{\epsilon}{2}\Big|\frac{\partial\chi_\epsilon}{\partial x_2}\Big|^2=\mathcal O( \epsilon^{4/3}\zeta_\epsilon^3),$$
\begin{multline*}
-\int_{D^1_\epsilon}\frac{\mu}{2\epsilon}\chi^2_\epsilon=-\int_{D^1_\epsilon}\frac{\mu(0,x_2)}{2\epsilon}\chi^2_\epsilon+\mathcal O(\epsilon\zeta_\epsilon^2)=-\int_{|x_1|<\zeta_\epsilon, |x_2|<\tau_\epsilon }\frac{l_\epsilon^2(x_2)}{\mu(0,x_2)}\frac{\mu(0,x_2)\tilde \chi^2}{2}+\mathcal O(\epsilon\zeta_\epsilon^2)\\=-\int_{|x_1|<\zeta_\epsilon, |x_2|<\tau_\epsilon }\frac{\mu(0,x_2)\tilde \chi^2}{2}+\mathcal O( \epsilon^{4/3}\zeta_\epsilon^3),
\end{multline*}
$$\int_{D^1_\epsilon }\frac{\chi_\epsilon^4}{4\epsilon}=\int_{|x_1|<\zeta_\epsilon, |x_2|<\tau_\epsilon }\frac{l_\epsilon^4(x_2)}{(\mu(0,x_2))^2}\frac{\tilde \chi^4}{4}=\int_{|x_1|<\zeta_\epsilon ,|x_2|<\tau_\epsilon}\frac{\tilde \chi^4}{4}+\mathcal O( \epsilon^{4/3}\zeta_\epsilon^3).$$
Gathering the previous results,
it follows that 
$$\lim_{\epsilon\to0}\int_{D_\epsilon^3}\big(\frac{\epsilon}{2}|\nabla \chi_\epsilon|^2+\frac{(\chi_\epsilon^2-\mu)^2}{4\epsilon}\big)=\int_{-\rho}^\rho\int_{\R}\Big(\frac{1}{2}\Big|\frac{\partial\tilde \chi}{\partial x_1}\Big|^2+\frac{(\tilde \chi^2-\mu(0,x_2))^2}{4}\Big)\dd x_1\dd x_2=\frac{2\sqrt{2}}{3}\int_{-\rho}^\rho (\mu(0,x_2))^{3/2}\dd x_2.$$
Finally, in view of what precedes we deduce that $$\limsup_{\epsilon\to 0}\mathcal{E}_{\epsilon,a}(v_{\epsilon,a})\leq\lim_{\epsilon\to 0}\mathcal E_{\epsilon,a}(\chi_\epsilon)=\frac{2\sqrt{2}}{3}\int_{-\rho}^\rho (\mu_{\mathrm{rad}}(r))^{3/2}\dd r -a\int_{D(0;\rho)}|f_1|\sqrt{\mu}.$$
 \end{proof}
At this stage, we are going to compute a lower bound of $\mathcal{E}_{\epsilon,a}(v_{\epsilon,a})$ (cf. \eqref{fatt}). This computation reduces to the one 
dimensional problem studied in \cite{panayotis_1}. For every $x_2\in (-\rho,\rho)$ fixed, we consider the restriction of the energy to the line $\{(t,x_2): t\in \R\}$:
\begin{equation}\label{funct 00}
E^{x_2}(\phi)=\int_{\R}\left(\frac{\epsilon}{2}|\phi'(t)|^2-\frac{1}{2\epsilon}\mu(t,x_2)\phi^2(t)+\frac{1}{4\epsilon}|\phi(t)|^4-a f_1(t,x_2)\phi(t)\right)\dd t, \ \phi \in H^1(\R).
\end{equation}
We recall (cf. \cite{panayotis_1}) that there exists $\psi^{x_2}_{\epsilon,a} \in H^1(\R)$ such that 
$E^{x_2}(\psi^{x_2}_{\epsilon,a})=\min_{H^1(\R)} E^{x_2}$, and moreover setting
$$\mathcal{E}^{x_2}(\phi):=E^{x_2}(\phi)+\int_{|t|<\sqrt{\rho^2-x_2^2}}\frac{\mu^2(t,x_2)}{4\epsilon}\dd t,$$
$$a_*(x_2):=\inf_{t \in (-\sqrt{\rho^2-x_2^2},0]}\frac{\sqrt{2}(\mu(t,x_2))^{3/2}}{3\int_{-\sqrt{\rho^2-x_2^2}}^t 
|f_1(t,x_2)|\sqrt{\mu(t,x_2)} },$$ 
and
$$a^*(x_2):=\sup_{t\in[-\sqrt{\rho^2-x_2^2},0)} 
\frac{\sqrt{2}\big((\mu(0,x_2))^{3/2}-(\mu(t,x_2))^{3/2}\big)}{3\int_t^0|f_1(t,x_2)|\sqrt{\mu(t,x_2)}},$$
we have
\begin{equation}\label{recal1}
 \lim_{\epsilon\to 0} \mathcal{E}^{x_2}_{\epsilon,a}(\psi^{x_2}_{\epsilon,a})=0, \forall x_2\in (-\rho,\rho), \forall a\in(0, a_*(x_2)),
\end{equation}
and
\begin{equation}\label{recal2}
\lim_{\epsilon\to 0} \mathcal{E}^{x_2}_{\epsilon,a}(\psi^{x_2}_{\epsilon,a})=\frac{2\sqrt{2}}{3}(\mu(0,x_2))^{3/2}-a
\int_{|t|<\sqrt{\rho^2-x_2^2}} |f_1(t,x_2)| \sqrt{\mu(t,x_2)}\dd t, \forall x_2\in (-\rho,\rho), \forall a\in(a^*(x_2),\infty). 
\end{equation}
Also note that $0<a_*=\inf_{x_2\in (-\rho,\rho)}a_*(x_2)\leq a_*(x_2)\leq a^*(x_2)\leq a^*=\sup_{x_2\in (-\rho,\rho)} a^*(x_2)$, for every $x_2\in (-\rho,\rho)$.
In view of these results we claim:

\begin{lemma}\label{bbnmbispr}We have
\begin{equation}\label{fattlem1}
\lim_{\epsilon\to 0}\int_{\R^2}\epsilon\Big| \frac{\partial v_{\epsilon,a}}{\partial x_2}\Big|^2=0 \text{ when } a\in (0,a_*)\cup (a^*,\infty), 
 \end{equation}
\begin{equation}\label{fatnew1}
\lim_{\epsilon\to 0}\int_{-\rho}^\rho 
|\mathcal{E}^{x_2}_{\epsilon,a}(v_{\epsilon,a}(\cdot,x_2))|\dd x_2=0 \text{ when } a\in (0,a_*),
\end{equation}
and
\begin{equation}\label{fatnew2}
\lim_{\epsilon\to 0}\int_{-\rho}^\rho 
\Big|\mathcal{E}^{x_2}_{\epsilon,a}(v_{\epsilon,a}(\cdot,x_2))     -       \frac{2\sqrt{2}}{3}(\mu(0,x_2))^{3/2}+a
\int_{|t|<\sqrt{\rho^2-x_2^2}}  |f_1(t,x_2)| \sqrt{\mu(t,x_2)}\dd t\Big| \dd x_2=0\text{ when } a\in (a^*,\infty).  
 \end{equation}
\end{lemma}
\begin{proof}
It is clear that 
$\mathcal E_{\epsilon,a}(v_{\epsilon,a})=\frac{\epsilon}{2}\int_{\R^2}|v_{x_2}|^2+\int_{-\rho}^\rho 
\mathcal{E}^{x_2}_{\epsilon,a}(v_{\epsilon,a}(\cdot,x_2))\dd x_2 +\int_{|x_2|>\rho}\frac{v^2(v^2-2\mu)}{4\epsilon}- a \int_{|x_2|>\rho}f_1 v$.
We are going to examine each of these integrals. 
In view of Theorem \ref{theorem 1} (iii), we have by dominated convergence
\begin{equation}\label{cvdomin}
 \lim_{\epsilon\to 0}\int_{|x_2|>\rho}f_1 v=0.
\end{equation}
On the other hand, since $\mathcal{E}^{x_2}_{\epsilon,a}(v_{\epsilon,a}(\cdot,x_2))\geq 
\mathcal{E}^{x_2}(\psi^{x_2}_{\epsilon,a})$, and $\mathcal{E}^{x_2}_{\epsilon,a}(v_{\epsilon,a}(\cdot,x_2))$ is uniformly bounded from below on $(-\rho,\rho)$,
it follows from \eqref{recal1}, \eqref{recal2}, and Fatou's Lemma that 
\begin{align}\label{fatt}
\liminf_{\epsilon\to 0}\int_{-\rho}^\rho 
\mathcal{E}^{x_2}_{\epsilon,a}(v_{\epsilon,a}(\cdot,x_2))\dd x_2&\geq\int_{-\rho}^\rho  \liminf_{\epsilon\to 0}
\mathcal{E}^{x_2}_{\epsilon,a}(v_{\epsilon,a}(\cdot,x_2))\dd x_2
\geq\int_{-\rho}^\rho  \liminf_{\epsilon\to 0}\mathcal{E}^{x_2}(\psi^{x_2}_{\epsilon,a})\nonumber \\
&\geq\begin{cases}0 &\text{ when } a\in (0,a_*),\\
\frac{2\sqrt{2}}{3}\int_{-\rho}^\rho (\mu_{\mathrm{rad}}(r))^{3/2}\dd r
-a\int_{D(0;\rho)}|f_1|\sqrt{\mu} &\text{ when } a\in (a^*,\infty).  
\end{cases}
 \end{align}
Next, we utilize \eqref{quaq1}, \eqref{cvdomin}, and \eqref{inegaaa}, to obtain 
\begin{align}\label{fattcal}
 \limsup_{\epsilon\to 0}\int_{-\rho}^\rho 
\mathcal{E}^{x_2}_{\epsilon,a}(v_{\epsilon,a}(\cdot,x_2))\dd x_2&\leq \limsup_{\epsilon\to 0}\mathcal{E}_{\epsilon,a}(v_{\epsilon,a})\nonumber \\
&\leq\begin{cases}0 &\text{ when } a\in (0,a_*),\\
\frac{2\sqrt{2}}{3}\int_{-\rho}^\rho (\mu_{\mathrm{rad}}(r))^{3/2}\dd r
-a\int_{D(0;\rho)}|f_1|\sqrt{\mu} &\text{ when } a\in (a^*,\infty).  
\end{cases}
 \end{align}
Combining \eqref{fatt} with \eqref{fattcal}, we deduce that  
\begin{equation}\label{fattlem}
\lim_{\epsilon\to 0}\int_{-\rho}^\rho 
\mathcal{E}^{x_2}_{\epsilon,a}(v_{\epsilon,a}(\cdot,x_2))\dd x_2=\begin{cases}0 &\text{ when } a\in (0,a_*),\\
\frac{2\sqrt{2}}{3}\int_{-\rho}^\rho (\mu_{\mathrm{rad}}(r))^{3/2}\dd r
-a\int_{D(0;\rho)}|f_1|\sqrt{\mu} &\text{ when } a\in (a^*,\infty),
\end{cases}
 \end{equation}
from which \eqref{fattlem1} follows. For a.e. $x_2\in(-\rho,\rho)$, we also obtain (respectively when $a\in (0,a_*)$ and  $a\in(a^*,\infty)$), that 
 \begin{equation}\label{fattlem2}
\liminf_{\epsilon\to 0} \mathcal{E}^{x_2}_{\epsilon,a}(v_{\epsilon,a}(\cdot,x_2))=\begin{cases}0,\\
 \frac{2\sqrt{2}}{3}(\mu(0,x_2))^{3/2}-a
\int_{|t|<\sqrt{\rho^2-x_2^2}} |f_1(t,x_2)| \sqrt{\mu(t,x_2)}\dd t, 
\end{cases}
 \end{equation}
thus
 \begin{equation}\label{fattlem23}\begin{cases}
\lim_{\epsilon\to 0} \min\big[\mathcal{E}^{x_2}_{\epsilon,a}(v_{\epsilon,a}(\cdot,x_2)),0\big]=0,\\
\lim_{\epsilon\to 0} \min\big[\mathcal{E}^{x_2}_{\epsilon,a}(v_{\epsilon,a}(\cdot,x_2))- \frac{2\sqrt{2}}{3}(\mu(0,x_2))^{3/2}+a
\int_{|t|<\sqrt{\rho^2-x_2^2}} |f_1(t,x_2)| \sqrt{\mu(t,x_2)}\dd t,0\big]=0,
\end{cases}
 \end{equation}
and by dominated convergence
\begin{equation}\label{fattlem24}\begin{cases}
\lim_{\epsilon\to 0}\int_{-\rho}^\rho \min\big[\mathcal{E}^{x_2}_{\epsilon,a}(v_{\epsilon,a}(\cdot,x_2)),0\big]\dd x_2=0,\\
\lim_{\epsilon\to 0} \int_{-\rho}^\rho\min\big[\mathcal{E}^{x_2}_{\epsilon,a}(v_{\epsilon,a}(\cdot,x_2))- \frac{2\sqrt{2}}{3}(\mu(0,x_2))^{3/2}+a
\int_{|t|<\sqrt{\rho^2-x_2^2}}  |f_1(t,x_2)| \sqrt{\mu(t,x_2)}\dd t,0\big]\dd x_2=0.
\end{cases}
\end{equation}
Combining \eqref{fattlem} with \eqref{fattlem24}, we conclude that \eqref{fatnew1} and \eqref{fatnew2} hold.
 \end{proof}
The proof of Theorem \ref{thcv} (ii) will follow from

\begin{lemma}\label{conclth2}For fixed $a\in (0,a_*)$, we have
\begin{equation}\label{qaz1}
\lim_{\epsilon\to 0}\int_{\R^2}f_1v_{\epsilon,a}=0,  
\end{equation}
and 
\begin{equation}\label{qaz2}
\lim_{\epsilon\to 0}\Big(\int_{\R^2}\frac{\epsilon}{2}|\nabla v_{\epsilon,a}|^2+\int_{|x|<\rho}
\frac{(v_{\epsilon,a}^2-\mu)^2}{4\epsilon}\Big)=0.
\end{equation}
\end{lemma}
\begin{proof}
Given a sequence $\epsilon_n\to 0$, we are going to show that we can extract a subsequence $\epsilon'_n\to 0$ such that $\lim_{n\to\infty}\int_{\R^2}f_1v_{\epsilon'_n,a}=0$. This will prove \eqref{qaz1}.
According to \eqref{fattlem1} and \eqref{fatnew1}, there exists a negligible set $N\subset (-\rho,\rho)$ such that for a subsequence called $\epsilon'_n$, and for every $x_2\in (-\rho,\rho)\setminus N$, we have
\begin{equation}\label{claa1}
\lim_{n\to \infty}\int_{\R}\epsilon'_n\Big| \frac{\partial v_n}{\partial x_2}(t,x_2)\Big|^2\dd t=0, 
 \end{equation}
and
\begin{equation}\label{claa2}
\lim_{n\to \infty}\mathcal{E}^{x_2}_{\epsilon'_n,a}(v_n(\cdot,x_2))=0, 
\end{equation}
where we have set $v_n=v_{\epsilon'_n,a}$.
Our claim is that 
\begin{equation}\label{claiml1}
\lim_{n\to\infty}\int_{\R} f_1(t,x_2)v_n(t,x_2) \dd t=0, \ \forall x_2\in (-\rho,\rho)\setminus N.
\end{equation}
From \eqref{claa1} and \eqref{claa2}, it follows that given $x_2\in (-\rho,\rho)\setminus N$ and $\gamma\in(0,\sqrt{\rho_2^2-x_2^2})$, there exists $ \bar n(x_2,\gamma)$ such that 
\begin{equation}\label{zerovv}
n\geq\bar  n(x_2,\gamma),\ |t|< \gamma\Rightarrow v_n(t,x_2)\neq 0.
\end{equation}
Indeed, otherwise we can find a subsequence $n_k$ and a sequence $(-\gamma,\gamma)\ni t_k\to t_0$ such that $v_{n_k}(t_k,x_2)=0$. Then, proceeding as in \cite[Proof of Theorem 1.1, Step 6]{panayotis_1} we obtain that $\liminf_{k\to \infty}\mathcal{E}^{x_2}_{\epsilon'_{n_k},a}(v_{n_k}(\cdot,x_2))>0$, which contradicts \eqref{claa2}. Next, for fixed $t \in (-\gamma,\gamma)$, we set $\tilde v_n(s):=v_n(t+\epsilon'_ns_1,x_2+\epsilon'_n s_2)$, and proceeding as in the proof of Theorem \ref{theorem 1} (i) above, we can see that 
$\tilde v_n$ converges in $C^2_{\mathrm{ loc}}(\R^2)$ to a minimal solution $\tilde V$ of the equation $\Delta \tilde V +(\mu(t,x_2)-\tilde V^2)\tilde V=0$.
If $\tilde V(s)=\sqrt{\mu(t,x_2)}\tanh(\sqrt{\mu(t,x_2)/2}(s-s_0)\cdot \nu)$, for some unit vector $\nu =(\nu_1,\nu_2)\in \R^2$, and some $s_0\in \R^2$, then \eqref{zerovv} excludes the case where $\nu_1\neq 0$, while \eqref{claa1} excludes the case where $\nu_2\neq 0$. Thus, $\tilde V(s)\equiv \pm \sqrt{\mu(t,x_2)}$, and in particular $\lim_{n\to\infty}|v_n(t,x_2)|=\sqrt{\mu(t,x_2)}$. Finally, given $\delta>0$, we choose $\gamma$ such that $2(\sqrt{\rho^2-x_2^2}-\gamma)\|f_1\| _{L^\infty}\sup_n  \|v_n\| _{L^\infty}<\delta/2$, and since $\lim_{n\to\infty}\big|\int_{-\gamma}^\gamma f_1(t,x_2)v_n(t,x_2)\dd t\big|=\lim_{n\to\infty}\big|\int_{-\gamma}^\gamma f_1(t,x_2)|v_n(t,x_2)|\dd t\big|=0$, we deduce that $$\Big|\int_{|t|<\sqrt{\rho^2-x_2^2}} f_1(t,x_2)v_n(t,x_2)\dd t\Big|<\delta$$ provided that $n$ is big enough. This proves that  $\lim_{n\to\infty}\int_{|t|<\sqrt{\rho2-x_2^2}}f_1(t,x_2)v_n(t,x_2)\dd t=0$, and recalling that $\lim_{n\to\infty}\int_{|t|>\sqrt{\rho2-x_2^2}}f_1(t,x_2)v_n(t,x_2)\dd t=0$ in view of Theorem \ref{theorem 1} (iii), we have established \eqref{claiml1}. Then, we conclude that $\lim_{n\to \infty}\int_{|x_2|<\rho}f_1v_n=0$ by dominated convergence, and since $\lim_{n\to \infty}\int_{|x_2|>\rho}f_1v_n=0$ by Theorem \ref{theorem 1} (iii), we have proved \eqref{qaz1}. The limit in \eqref{qaz2} follows from \eqref{qaz1} and \eqref{quaq1}.
\end{proof}
\begin{proof}[Conclusion of the proof of Theorem \ref{thcv} (ii)]
We first show that when $a\in(0,a_*)$, we have $Z \subset \{|x|=\rho\} \cup \{x_1=0, |x_2|\geq \rho\}$. 
Assume by contradiction that there exist a sequence $\epsilon_n\to 0$, and a sequence $\bar x_n\to x_0 \in D(0;\rho_0)$, with $\rho_0<\rho$, such that $v_n:=v_{\epsilon_n,a}$ vanishes at $\bar x_n$. By Lemma \ref{smooth} (i), we know that $\bar x_n$ belongs to a smooth branch of zeros that we called $Z_{\epsilon_n}$. Let $D_1(n)=\{x_1: (x_1,x_2)\in Z_{\epsilon_n}\}$, $D_2(n)=\{x_2: (x_1,x_2)\in Z_{\epsilon_n}\}$, and for $i=1,2$, let $\delta_i(n)=\mathcal L^1(D_i(n))$, where $\mathcal L$ denotes the Lebesgue measure. Since by Lemma \ref{smooth} (ii), we have $\frac{(v_{n}^2(x)-\mu(x))^2}{4\epsilon_n}\geq \frac{9\mu_{\mathrm{rad}}^2(\rho_0)}{4^3\epsilon_n}$, for $x \in \cup_{z\in Z_{\epsilon_n} }D(z; l\epsilon_n)$,
it follows that $ \frac{9\mu_{\mathrm{rad}}^2(\rho_0)l}{4^3}\delta_i(n)\leq \int_{|x|<\rho}\frac{(v_{n}^2-\mu)^2}{4\epsilon_n}$, and thus $\lim_{n \to\infty}\delta_i(n)=0$, in view of \eqref{qaz2}. This implies in particular that $\bar x_n$ belongs to a smooth Jordan curve $\Gamma_n\subset D(0;\rho_0)$. Let $\omega_n$ be the open set bounded by $\Gamma_n$, let $\nu_n(z)$ be the outer unit normal vector at $z\in\Gamma_n$, and let us define the open set $\Omega_n=\{x\in\R^2: d(x,\overline \omega_n)<\lambda \epsilon_n\}$, where $d$ stands for the Euclidean distance, and $\lambda$ is the constant defined in Lemma \ref{smooth}. 
As previously we set $\tilde D_1(n)=\{x_1: (x_1,x_2)\in \Gamma_{n}\}$, $\tilde D_2(n)=\{x_2: (x_1,x_2)\in \Gamma_{n}\}$, and for $i=1,2$, $\tilde \delta_i(n)=\mathcal L^1(\tilde D_i(n))$.
By Lemma \ref{smooth} (iii), we have either $v_n\geq \frac{3}{4}\sqrt{\mu_{\mathrm{rad}}(\rho_0)}$ or $v_n\leq -\frac{3}{4}\sqrt{\mu_{\mathrm{rad}}(\rho_0)}$ on $\partial \Omega_n$. Assuming without loss of generality that $v_n\geq \frac{3}{4}\sqrt{\mu_{\mathrm{rad}}(\rho_0)}$ on $\partial \Omega_n$, we introduce the comparison function
\begin{equation}\label{chi1}
\chi_n(x)=\begin{cases}
v_n(x) &\text{for } x\in\R^2\setminus \Omega_n\\
\max(|v_n(x) |,  \sqrt{\mu_{\mathrm{rad}}(\rho_0)}/2)&\text{for } x\in\Omega_n,
\end{cases}
\end{equation} 
and notice that $|\mu-\chi_n^2|\leq|\mu-v_n^2|$. Setting $S_n:=\{x: d(x, \Gamma_n)< l\epsilon_n\}\subset \Omega_n$, it is clear that 
$\mathcal L^2(S_n)\geq \tilde\delta_i(n)l\epsilon_n$ for $i=1,2$. In addition, according to Lemma \ref{smooth} (ii) and (iv), the inequalities $|v_n|\leq \sqrt{\mu_{\mathrm{rad}}(\rho_0)}/2$, and $\epsilon_n |\nabla v_n|\geq \lambda'$ hold on $S_n$.
Finally, we also notice that Lemma \ref{smooth} (iv) implies that $\tilde \delta_i(n)\geq \lambda\epsilon_n$. Gathering these results we reach the following contradiction
\begin{align}
\mathcal{E}_{\epsilon_n,a}(\chi_{n})-\mathcal{E}_{\epsilon_n,a}(v_{n})&\leq -\frac{\epsilon_n}{2}\int_{|v_n|\leq  \sqrt{\mu_{\mathrm{rad}}(\rho_0)}/2}|\nabla v_n|^2+a\int_{\Omega_n}f(v_n-\chi_n)\nonumber\\
&\leq-\frac{|\lambda'|^2}{2\epsilon_n}\mathcal L^2(S_n)+K\mathcal L^2(\Omega_n), \text{ where $K>0$ is a constant} \nonumber\\
&\leq-\frac{|\lambda'|^2l}{2}\tilde \delta_1(n)+K(\tilde \delta_1(n)+2\lambda \epsilon_n)(\tilde \delta_2(n)+2\lambda \epsilon_n)\nonumber\\
&\leq\Big( 9K\tilde \delta_2(n)-\frac{|\lambda'|^2l}{2}\Big)\tilde \delta_1(n)<0, \text{ for $n$ large enough}. 
\end{align}
This proves that there are no limit points of the zeros of $v$ in $D(0;\rho)$. In view of Theorem \ref{theorem 1} (iii) we deduce that
$Z \subset \{|x|=\rho\} \cup \{x_1=0, |x_2|\geq \rho\}$. Another consequence is that given $\rho_0\in(0,\rho)$, there exists $\epsilon_0>0$ such that when $\epsilon\in (0,\epsilon_0)$, the minimizer $v_{\epsilon,a}$ does not vanish on $D(0;\rho_0)$.
Up to change of $v(x_1,x_2)$ by $-v(-x_1,x_2)$, we may assume that $v_{\epsilon,a}>0$ on $D(0;\rho_0)$. Then, in view of Theorem \ref{theorem 1} (iii) we have
$ \{x_1<0, |x|=\rho\} \cup \{x_1=0, |x_2|\geq \rho\}\subset Z $. Finally, the limit in \eqref{cvn2} follows from Theorem \ref{theorem 1} (i), in the case where $|x|<\rho$. On the other hand, for fixed $x$ such that $|x|\geq\rho$, the rescaled minimizers
$\tilde v(s)= v(x+s\epsilon)$ converge to a bounded solution $\tilde V$ of the equation $\Delta \tilde V(s)+(\mu(x)-\tilde V^2(s))\tilde V(s)=0$. As in the proof of Theorem \ref{theorem 1} (iii),
the associated potential $W(u)=\frac{u^4}{4}-\frac{\mu(x)}{2}u^2$ is stricly convex, thus $\tilde V$ satisfies $W'(\tilde V)=0$, i.e. $\tilde V=0$. 
\end{proof}

Now we establish the analog of Lemma \ref{conclth2} in the case where $a>a^*$, to complete the proof of Theorem \ref{thcv} (iii).

\begin{lemma}\label{conclth2aa}For fixed $a\in (a^*,\infty)$, and for every $\gamma\in (0,\rho)$, we have
\begin{equation}\label{qaz1aa}
\lim_{\epsilon\to 0}\int_{\R^2}f_1v_{\epsilon,a}=\int_{D(0;\rho)}|f_1|\sqrt{\mu},  
\end{equation}
\begin{equation}\label{fattlem1aa}
\lim_{\epsilon\to 0}\int_{|x|<\rho, |x_1|<\gamma}\Big(\frac{\epsilon}{2}\Big| \frac{\partial v_{\epsilon,a}}{\partial x_1}\Big|^2+\frac{(v_{\epsilon,a}^2-\mu)^2}{4\epsilon}\Big)=\int_{-\rho}^\rho \frac{2\sqrt{2}}{3}(\mu(0,x_2))^{3/2}\dd x_2,
 \end{equation}
\begin{equation}\label{qaz2aa}
\lim_{\epsilon\to 0}\int_{|x|<\rho, |x_1|>\gamma}\Big(\frac{\epsilon}{2}| \nabla  v_{\epsilon,a}|^2+\frac{(v_{\epsilon,a}^2-\mu)^2}{4\epsilon}\Big)=0.
\end{equation}
\end{lemma}
\begin{proof}
Given a sequence $\epsilon_n\to 0$, we are going to show that we can extract a subsequence $\epsilon'_n\to 0$ such that $\lim_{n\to\infty}\int_{\R^2}f_1v_{\epsilon'_n,a}=\int_{D(0;\rho)}|f_1|\sqrt{\mu}$, and
$$\lim_{n\to\infty}\int_{|x|<\rho, |x_1|<\gamma}\Big(\frac{\epsilon'_n}{2}\Big| \frac{\partial v_{\epsilon'_n,a}}{\partial x_1}\Big|^2+\frac{(v_{\epsilon'_n,a}^2-\mu)^2}{4\epsilon'_n}\Big)=\int_{-\rho}^\rho \frac{2\sqrt{2}}{3}(\mu(0,x_2))^{3/2}\dd x_2.$$ This will prove \eqref{qaz1aa} and \eqref{fattlem1aa}.
According to \eqref{fattlem1} and \eqref{fatnew2}, there exists a negligible set $N\subset (-\rho,\rho)$ such that for a subsequence called $\epsilon'_n$, and for every $x_2\in (-\rho,\rho)\setminus N$, we have
\begin{equation}\label{claa1aa}
\lim_{n\to \infty}\int_{\R}\epsilon'_n\Big| \frac{\partial v_n}{\partial x_2}(t,x_2)\Big|^2\dd t=0, 
 \end{equation}
and
\begin{equation}\label{claa2aa}
\lim_{n\to \infty}\mathcal{E}^{x_2}_{\epsilon'_n,a}(v_n(\cdot,x_2))= \frac{2\sqrt{2}}{3}(\mu(0,x_2))^{3/2}-a
\int_{|t|<\sqrt{\rho^2-x_2^2}}  |f_1(t,x_2)| \sqrt{\mu(t,x_2)}\dd t, 
\end{equation}
where we have set $v_n=v_{\epsilon'_n,a}$.
Our claim is that 
\begin{equation}\label{claiml1aa}
\lim_{n\to\infty}\int_{\R} f_1(t,x_2)v_n(t,x_2) \dd t=\int_{|t|<\sqrt{\rho^2-x_2^2}}  |f_1(t,x_2)| \sqrt{\mu(t,x_2)}\dd t, \ \forall x_2\in (-\rho,\rho)\setminus N.
\end{equation}
From \eqref{claa1aa} and \eqref{claa2aa}, it follows that given $x_2\in (-\rho,\rho)\setminus N$ and $\gamma\in(0,\rho)$, there exists $ \bar n(x_2,\gamma)$ such that 
\begin{equation}\label{zerovvaa}
n\geq\bar  n(x_2,\gamma),\ \gamma<|t|<\rho+1\Rightarrow v_n(t,x_2)\neq 0.
\end{equation}
Indeed, otherwise we can find a subsequence $n_k$ and a sequence $(-\rho-1,-\gamma)\cup(\gamma,\rho+1)\ni t_k\to t_0$ such that $v_{n_k}(t_k,x_2)=0$. Then, proceeding as in \cite[Proof of Theorem 1.1, Step 6]{panayotis_1} we obtain that $\liminf_{k\to \infty}\mathcal{E}^{x_2}_{\epsilon'_{n_k},a}(v_{n_k}(\cdot,x_2))> \frac{2\sqrt{2}}{3}(\mu(0,x_2))^{3/2}-a
\int_{|t|<\sqrt{\rho^2-x_2^2}}  |f_1(t,x_2)| \sqrt{\mu(t,x_2)}\dd t$, which contradicts \eqref{claa2aa}. Thus, \eqref{zerovvaa} holds, and actually in view of Theorem \ref{theorem 1} (iii) we have
\begin{equation}\label{zerovvaabis}
n\geq\bar  n(x_2,\gamma),\ \gamma<t<\rho+1\Rightarrow v_n(t,x_2)> 0, \text{ and }n\geq\bar  n(x_2,\gamma),\ -\rho-1<t<-\gamma \Rightarrow v_n(t,x_2)< 0.
\end{equation}
 Next, for fixed $t \in (-\sqrt{\rho^2-x_2^2},-\gamma)\cup(\gamma,\sqrt{\rho^2-x_2^2})$, we set $\tilde v_n(s):=v_n(t+\epsilon'_ns_1,x_2+\epsilon'_n s_2)$, and proceeding as in the proof of Lemma \ref{conclth2}, we can see that 
 $\lim_{n\to\infty}v_n(t,x_2)=\sqrt{\mu(t,x_2)}$ for $t\in (\gamma,\sqrt{\rho^2-x_2^2})$, while 
 $\lim_{n\to\infty}v_n(t,x_2)=-\sqrt{\mu(t,x_2)}$ for $t\in (-\sqrt{\rho^2-x_2^2},-\gamma)$.
Then, by repeating the arguments in the proof of Lemma \ref{conclth2}, our claim 
\eqref{claiml1aa} follows. Finally, we conclude that $\lim_{n\to \infty}\int_{|x_2|<\rho}f_1v_n=\int_{D(0;\rho)}|f_1|\sqrt{\mu}$ by dominated convergence, and since $\lim_{n\to \infty}\int_{|x_2|>\rho}f_1v_n=0$ by Theorem \ref{theorem 1} (iii), we have established \eqref{qaz1aa}. Another consequence of \eqref{zerovvaabis} is that for every $x_2\in (-\rho,\rho)\setminus N$, there exists a sequence $\bar t_n\to 0$ such that $v_n(\bar t_n,x_2)=0$. 
Setting $\tilde v_n(s):=v_n(\bar t_n+\epsilon'_ns_1,x_2+\epsilon'_n s_2)$, we obtain as in Lemma \ref{conclth2}, that 
$\tilde v_n$ converges in $C^2_{\mathrm{ loc}}(\R^2)$ to $\tilde V(s)=\sqrt{\mu(0,x_2)}\tanh(\sqrt{\mu(0,x_2)/2}(s\cdot \nu))$, for some unit vector $\nu =(\nu_1,\nu_2)\in \R^2$. Again, \eqref{claa1aa} implies that $\nu=(1,0)$, and we refer to the detailed computation in \cite[Proof of Theorem 1.1, Step 6]{panayotis_1} to see that 
\begin{equation}\label{detail}
\liminf_{n\to\infty}\int_{|t|<\min(\gamma,\sqrt{\rho^2-x_2^2})}\Big(\frac{\epsilon'_n}{2}\Big| \frac{\partial v_{n}}{\partial x_1}(t,x_2)\Big|^2+\frac{(v_{n}^2(t,x_2)-\mu(t,x_2))^2}{4\epsilon'_n}\Big)\dd t\geq \frac{2\sqrt{2}}{3}(\mu(0,x_2))^{3/2}.
\end{equation}
Then, it follows from Fatou's Lemma that
\begin{equation}\label{detail2}
\liminf_{n\to\infty}\int_{|x|<\rho, |x_1|<\gamma}\Big(\frac{\epsilon'_n}{2}\Big| \frac{\partial v_{\epsilon'_n,a}}{\partial x_1}\Big|^2+\frac{(v_{\epsilon'_n,a}^2-\mu)^2}{4\epsilon'_n}\Big)\geq\int_{-\rho}^\rho \frac{2\sqrt{2}}{3}(\mu(0,x_2))^{3/2}\dd x_2.
\end{equation}
Finally, combining \eqref{detail2} with \eqref{qaz1aa} and \eqref{quaq1}, we deduce \eqref{fattlem1aa} and \eqref{qaz2aa}.
\end{proof}
\begin{proof}[Conclusion of the proof of Theorem \ref{thcv} (iii)]
Proceeding as in the conclusion of the proof of Theorem \ref{thcv} (ii), we show that there are no limit points of the zeros of $v$ in the set $D(0;\rho)\cap\{(x_1,x_2): |x_1|> \gamma\}$, where $\gamma>0$ is small. 
As a consequence, given $\rho_0\in(0,\rho)$, there exists $\epsilon_0>0$ such that when $\epsilon\in (0,\epsilon_0)$, the minimizer $v_{\epsilon,a}$ is positive on $D(0;\rho_0)\cap\{(x_1,x_2): x_1> \gamma\}$. 
Let $K\subset (\gamma,\rho+1) \times(-\rho_0,\rho_0)$ be a compact set.
Our claim is that there exists $\epsilon_K>0$ such that when $\epsilon\in (0,\epsilon_K)$, the minimizer $v_{\epsilon,a}$ is positive on $K$. To prove this claim we assume by contradiction that there exist a sequence $\epsilon_n\to 0$, and a sequence 
$K\ni x_n\to x_0$ such that $v_n(x_n)\leq 0$, where we have set $v_n:=v_{\epsilon_n,a}$. Having a closer look at the proof of Lemma \ref{conclth2aa} (cf. in particular \eqref{zerovvaabis}), we can find $\rho_1\in(0,\rho_0)$ such that 
$K\subset (\gamma,\rho+1) \times(-\rho_1,\rho_1)$, and $\bar  n(\rho_1,\gamma)$ such that for $n\geq\bar  n(\rho_1,\gamma)$, and $t\in(\gamma,\rho+1)$, we have $v_n(t,\pm\rho_1)>0$. Next, in view of Theorem \ref{theorem 1} (iii), we also obtain that $v_n(\rho+1,s)>0$ for every $s\in [-\rho_1,\rho_1]$, provided that $n$ is large enough. Gathering these results, it follows that there exists $n_K$ such that for every $n\geq n_K$, $v_n$ is positive on the boundary of the rectangle $R:=  (\gamma,\rho+1)\times(-\rho_1,\rho_1)$. In addition, for $n\geq n_K$, $v_n$ cannot take negative values in $R$, since otherwise we would have $E(|v_n|,R)<E(v_n,R)$ in contradiction with the minimality of $v_n$. Thus, $v_n$ has a local minimum at $x_n$ for $n\geq n_K$, and \eqref{ode} implies that $0\leq \epsilon^2\Delta v_n(x_n)=-\epsilon a f_1(x_n)\in(-\infty,0)$, which is a contradiction.
This establishes our claim, and now in view of Theorem \ref{theorem 1} (iii) it is clear that $Z=\{x_1=0\}$. Finally, the limit in \eqref{cvn1} is established as in the conclusion of the proof of Theorem \ref{thcv} (ii). To prove the limit in \eqref{cvn00}, we proceed as in Lemma \ref{conclth2aa}. There exists a subsequence $\epsilon'_n\to 0$, and a negligible set $N\subset(-\rho,\rho)$, such that \eqref{claa1aa} and \eqref{claa2aa} hold for every $x_2\in(-\rho,\rho)\setminus N$.
Now, let $\bar x_{n}=(\bar t_{\epsilon'_n,a},x_2)$ be a zero of $v_{n}$ with fixed ordinate $x_2\in(-\rho,\rho)\setminus N$, and set $\tilde v_n(s):= v(\bar x_n+\ve'_n s)$.
Then $\tilde v_n$ converges in the $C^2_{\mathrm{loc}}(\R)$ sense to $\tilde V(s)=\sqrt{\mu(0,x_2)}\tanh(\sqrt{\mu(0,x_2)/2}(s\cdot \nu))$ for some unit vector $\nu\in\R^2$, and \eqref{claa1aa} implies that $\nu=(\pm 1,0)$, while \eqref{claa2aa} implies that
for $n$ large enough $v_n$ has a unique zero with fixed ordinate $x_2$. Thus, $\nu=(1,0)$ and \eqref{cvn00} is established.
\end{proof} 
\end{proof}
%\nocite{*}

%\bibliography{biblio_29_08_2016}
%\bibliographystyle{amsplain}

\providecommand{\bysame}{\leavevmode\hbox to3em{\hrulefill}\thinspace}
\providecommand{\MR}{\relax\ifhmode\unskip\space\fi MR }
% \MRhref is called by the amsart/book/proc definition of \MR.
\providecommand{\MRhref}[2]{%
  \href{http://www.ams.org/mathscinet-getitem?mr=#1}{#2}
}
\providecommand{\href}[2]{#2}

\end{document}